\documentclass[a4paper,12pt]{amsart}
\usepackage{amsmath,amssymb,amscd,amsfonts,amsthm}
\usepackage{amsrefs}
\numberwithin{equation}{section}
\newcommand{\p}{\partial}
\newcommand{\vphi}{\varphi}
\newcommand{\om}{\omega}
\newcommand{\tri}{\triangle}
\newcommand{\eps}{\epsilon}
\newcommand{\thmref}[1]{Theorem~\ref{#1}}

\newcommand{\lemref}[1]{Lemma~\ref{#1}}
\newcommand{\corref}[1]{Corollary~\ref{#1}}

\newcommand{\Om}{\Omega}
\newcommand{\Na}{\nabla}
\newcommand{\cH}{\mathcal H}
\def\p{\partial}
\def\b{\beta}
\def\pbp{\sqrt{-1}\partial\bar\partial}

\def\osc{{\rm osc\,}}
\def\Aut{{\rm Aut}}

\newtheorem{theorem}{Theorem}[section]
\newtheorem{thm}[theorem]{Theorem}
\newtheorem{rem}[theorem]{Remark}
\newtheorem{defn}[theorem]{Definition}
\newtheorem{lemma}[theorem]{Lemma}
\newtheorem{lem}[theorem]{Lemma}

\newtheorem{proposition}[theorem]{Proposition}
\newtheorem{prop}[theorem]{Proposition}

\newtheorem{cor}[theorem]{Corollary}

\def\Re{\mathop{\rm Re}\nolimits} 
 
\def\Aut{\mathop{\rm Aut}\nolimits}
\def\tr{\mathop{\rm tr}\nolimits}

\def\dbar{\bar\partial}
\def\ddbar{\partial\bar\partial}
\def\d{\partial}
\def\la{\lambda}

\let\ol=\overline

\let\ep=\varepsilon
\let\vphi=\varphi 

\def\bC{{\mathbb C}}
\def\bR{{\mathbb R}}

\def\a{\alpha}

\title[Uniqueness of K\"ahler-Einstein cone metrics]
{Generalized Matsushima's theorem and K\"ahler-Einstein cone metrics}
\author{Long Li}
  \address{Department of Mathematics and Statistics, McMaster University,
1280 Main Street West, Hamilton, ON L8S 4K1, Canada}
  \email[Long Li]{lilong@math.mcmaster.ca}

\author{Kai Zheng}
  \address{Mathematics Institute, University of Warwick, Coventry, CV4 7AL, UK}
  \email[Kai Zheng]{K.Zheng@warwick.ac.uk}

\begin{document}
\maketitle

\begin{abstract}
In this paper, we prove Matsushima's theorem for K\"ahler-Einstein metrics on a Fano manifold with cone singularities along a smooth divisor
that is not necessarily proportional to the anti-canonical class. 
We then give an alternative proof of uniqueness of K\"ahler-Einstein cone metrics by the continuity method.
Moreover, our method provides an existence theorem of K\"ahler-Einstein cone metrics with respect to conic Ding functional.

\end{abstract}

\section{Introduction}

Let $X$ be a K\"ahler manifold, $D$ be a hypersurface and $L_D$ be the associated line bundle of $D$. We denote the regular part by $M:=X\setminus D$. We assume the cone angle $0<\b\leq 1$. 
We further assume $L_D$ is positive and $-(K_X + (1-\beta)L_D) > 0$ and consider the K\"ahler class 
\begin{align*}
\Om=-(K_X+(1-\beta)L_D).
\end{align*}
The automorphism of the pair $(X,D)$ is an automorphism of $X$ and fixs the divisor $D$, and all of these automorphisms of the pair consist of the group $Aut(X;D)$. 

A \emph{K\"ahler cone metric} of cone angle $2\pi\b$ along $D$,
is a closed positive $(1,1)$ current and a smooth K\"ahler metric on the regular part $M$.
{In a local holomorphic chart $\{U_p;z^1,\dots z^n\}$ around a point $p\in D$, its} K\"ahler form is quasi-isometric to the cone flat metric, which is
\begin{align*}
\om_{cone}
&:=\frac{\sqrt{-1}}{2}\b^2|z^1|^{2(\b-1)}dz^1\wedge
dz^{\bar 1}
+\sum_{2\leq j\leq n}dz^j\wedge dz^{\bar j}\, .
\end{align*}
Here $\{z^1,\dots z^n\}$ are the local defining functions of the hypersurface $D$ where $p$ locates.

The space of K\"ahler cone metrics associated to $\Omega$ is non-empty, it contains Donaldson's model metric (see \eqref{model cone} later).
We say a K\"ahler cone metric $\om_\vphi\in \Omega$ is the \emph{K\"ahler-Einstein cone metric} of cone angle $2\pi\b$ along $D$ if it satisfies the equation of currents,
\begin{align*}
Ric (\omega_\vphi)=\om_\vphi+2\pi(1-\b) [D].
\end{align*}

Our first theorem is to generalize Matsushima's theorem to K\"ahler-Einstein cone manifolds
\begin{thm}
Suppose the pair $(X, D)$ admits a K\"ahler-Einstein cone metric of angle $2\pi\b$. Then the Automorphism group $Aut(X;D)$ is reductive. 
\end{thm}

A more precise version can be found in \thmref{reductivityStatement}. 
In fact, we established a one-one correspondence between the holomorphic automorphism group 
and the complexification of the kernel of the following elliptic operator $$\Delta_{\theta} +1$$
at a K\"ahler-Einstein cone metric $\theta$.
And this one-one correspondence is stronger result than the reductivity of the automorphism group. 
Unlike the previous work in Fano case \cite{MR3264768}, we do not require that the K\"ahler class is proportional to the anti-canonical class, 
but certain positivity condition on the divisor is still needed. 
Moreover, it is worthy to mention that this theorem is proved by Kodaira-H\"ormander's $L^2$ techniques, 
but the Kodaira-Bochner formula for K\"ahler cone metrics is not clear to be true at this stage.

\begin{rem}
For the $klt$ - $pair$, Chen-Donaldson-Sun \cite{MR3264768} proved that the automorphism group is reductive. However, they required the uniqueness of weak K\"ahler-Einstein metrics in their proof. 
\end{rem}

\begin{rem}In \cite{MR3406526}, Cheltsov-Rubinstein also announced a result for extremal cone metrics, but their method is based on an expansion formula for edge metrics, which is very different from ours.
\end{rem}

Based on this reductivity result, we can extend Bando-Mabuchi's celebrated work \cite{MR946233} to conic setting and prove the uniqueness of K\"ahler-Einstein cone metris by applying the continuity path, which connects the K\"ahler-Einstein cone metric $\om_\vphi$ to a given K\"ahler cone metric $\om$. I.e. for any $t\in[0,1]$,
\begin{align*}
Ric(\om_{\vphi(t)})=t\om_{\vphi(t)}+(1-t)\om+2\pi(1-\b)[D].
\end{align*}
And we proved the following 

\begin{thm}
The K\"ahler-Einstein cone metric is unique up to automorphisms.
\end{thm}

The way to prove uniqueness is first to establish a continuity path connecting a general K\"ahler cone metric to our target, i.e. a K\"ahler-Einstein cone metric. 
The difficulties are to prove openness and closedness along the path in Donaldson's $C^{2,\a}_\b$ space: here openness on $[0,1)$ follows from a Bochner type formula with contradiction argument. Thus we are able to carry on the implicit function theorem on $[0,1)$ and the apriori estimates on $[0,\tau]$ for a small fixed $\tau>0$. 

Meanwhile, in order to prove closedness on $[\tau, 1]$, everything is boiled down to prove the zero order estimate (Section \ref{Zero order estimate title}) and the higher order estimates (Section \ref{Higher order estimate}). We first show that the zero order estimate of the continuity path on $[\tau, 1]$ requires only the uniform bound of the Sobolev constant, which is new even in the situation where all metrics are smooth. Then the Sobolev constant bound along the continuity path on $[\tau, 1]$ is proved by using an approximation of the continuity path. The approximation would have non-negative Ricci curvature and uniformly bounded diameter, which is an adaption of Theorem 1.1 in \cite{MR3264766} to our continuity path.

Finally, we need to establish a bifurcation technique (at $t=1$) under conic setting. 
In fact, this bifurcation technique for K\"ahler-Einstein cone metrics uses our generalized Matsushima's theorem. While, the computation of the second variation of the conic $I-J$ functional is more subtle than the smooth case.

\begin{rem}
The bifurcation method developed in Bando-Mabuchi \cite{MR946233} concerns the uniqueness of the smooth K\"ahler-Einstein metrics. Analogous result is Tian-Zhu \cite{MR1768112} in the context of K\"ahler-Ricci soliton.
\end{rem}

We would like to mention that 
our theorem generalises Bando-Mabuchi's result \cite{MR946233}, while 
fullfills the authors' projects \cite{Li1}\cite{Li2}\cite{Calamai-Zheng}\cite{Zheng}. The techniques built in this paper will be used in the sub-sequel papers on uniqueness of the constant scalar curvature K\"ahler metrics with cone singularities \cite{LZ2,LZ3}. 
In the beautiful work of Berndtsson \cite{MR3323577}, the uniqueness result for K\"ahler-Einstein cone metrics with normal crossing type divisors is proved. Our continuity method for cone metrics, togethor with an extension of Donaldson's $C^{2,\alpha}_{\beta}$ Schauder estimate for linear equations to normal crossing type divisors (which is believed to be true by many people), provides an alternative proof of Berndtsson's result. We also note that in the work of BBEGZ \cite{arXiv:1111.7158}, the uniqueness result was generalized to $klt$-pairs.


This continuity method approach indeed gives more geometric insights and simplified the proof on the equivalence between properness of the conic $Ding$-functional and the existence of K\"ahler-Einstein cone metrics, as a direct consequence of our method and estimates. 



\bigskip

$\mathbf{Acknowledgement}$: We are very grateful to Prof. Xiuxiong Chen. This paper was initiated during his invitation in Stony Brook in 2014. 
The first author would like to thank Prof. Mihai P\u aun, for pointing out a mistake in the first version.

The first author also wants to thank Prof. Ian Hambleton
and Prof. McKenzie Wang for their many supports. The second author would like to thank Prof. Bing Wang for his useful discussions.

This work was supported by the Engineering and Physical Sciences Research Council [EP/K00865X/1 to K.Z.]; and the European Commission [H2020-MSCA-IF-2015/703949(CFUC) to K.Z.].

\section{K\"ahler cone metrics}\label{skcp}
Let {$s$} be a global holomorphic section of $[D]$ and $h$ be a Hermitian metric on $[D]$. Once we are given a K\"ahler class $\Om$, we choose a smooth K\"ahler metric $\om_0$ in it.
{It
is shown in Donaldson} \cite{MR2975584} that, for sufficiently small $\delta>0$,
\begin{align}\label{model cone}
\om_D=\om_0+\delta \frac{\sqrt{-1}}{2} \p\bar\p |s|^{2\b}_{h}
\end{align}
is a K\"ahler cone metric.
Moreover, $\om_D$ is independent of the choices of $\om_0$, $h$, $\delta$ up to quasi-isometry. We call it \emph{model metric} in this paper.

The space of K\"ahler cone potentials $\mathcal H_{\b}$ consists of $\om_D$-psh functions of the K\"ahler cone metrics in $\Om$.




Now we present the function spaces which are introduced by Donaldson in
\cite{MR2975584}.
The H\"older space $C^{\a}_{\b}$  consists of those functions $f$
which are H\"older continuous with respect to a K\"ahler cone metric.
Note that according to this definition,
for any K\"ahler cone metric $\om\in C^\a_\b$, around the point $p\in D$,
we have a local normal coordinate such that $g_{ij}(p)=\delta_{ij}$.
\begin{defn}\label{defn: twice diff mixed derivatives functions}
The H\"older space $C^{2,\a}_{\b}$ is defined by
\begin{align*}
C^{2,\a}_{\b}
= \{f\; \vert \;  f, \p f, \p\bar\p f\in C_\b^{\a}\}\; .
\end{align*}
\end{defn}
Note that the $C^{2,\a}_\b$ space, since it concerns only with the mixed derivatives,
is different from the usual $C^{2,\a}$ H\"older space.
\subsection{Energy functionals}
Let $\om$ be a K\"ahler cone metric and $\omega_\vphi= \omega+i\ddbar\vphi$.
We denote the volume $V=\Om^n $.
The Aubin functions $I$ and $J$ could be defined on $\cH^{1,1}_\beta=C^{1,1}_\b\cap \cH_\b$  by
 \begin{align*}
I_\om(\vphi)&=\frac{1}{V}\int_{M}\vphi(\om^{n}-\om_{\vphi}^{n})
=\frac{\i}{V}\sum_{i=0}^{n-1}\int_{M}\p\varphi\wedge\bar\p\vphi\wedge\om^{i}\wedge\om_\vphi^{n-1-i},\\
J_\om(\vphi)&=\frac{\i}{V}\sum_{i=0}^{n-1}\frac{i+1}{n+1}
\int_{M}\p\vphi\wedge\bar\p\vphi\wedge\om^{i}\wedge\om^{n-1-i}_{\vphi}.
\end{align*} Note that the functionals $I$ and $J$  satisfy
the inequalities
\begin{align*}
\frac{1}{n+1}I\leq J\leq\frac{n}{n+1}I.
\end{align*}
The Lagrangian functional of the
Monge-Amp\`ere operator is
\begin{align}\label{i functional}
D_\om(\vphi)=\frac{1}{V}\int_M\vphi\om^n-J_\om(\vphi).
\end{align}
The derivative of $D_{\om}$ along a general path $\varphi_t\in\cH^{1,1}_\b$
is given by
\begin{align*}
\frac{d}{dt}D_\om(\vphi_t)=\frac{1}{V}\int_{M}\dot\vphi_t\,\om^n_{\vphi_t}.
\end{align*}
We could compute the explicit formula of $D_\om(\vphi)$ as the
following
\begin{align}
D_\om(\vphi)
&=\frac{1}{V}\sum_{i=0}^n\frac{n!}{(i+1)!(n-i)!}
\int_{M}\vphi\om^{n-i}\wedge(\pbp\vphi)^i\nonumber\\
&=\frac{1}{V}\int_M\vphi\om^n-\frac{\i}{V}\sum_{i=0}^{n-1}\frac{i+1}{n+1}
\int_{M}\p\vphi\wedge\bar\p\vphi\wedge\om^{i}\wedge\om_\vphi^{n-i-1}.
\label{eq:D}
\end{align}
Let $\cH_\b^0$ be the subspace of $\mathcal H_\b\cap C^{1,1}_\b$ with the normalization condition
$$\mathcal H_\b^0:=\{\vphi\in \mathcal H_\b\cap C^{1,1}_\b\vert D_\om(\vphi)=0\}.$$

\subsection{K\"ahler-Einstien cone metrics}




Recall that $D$ is a simple smooth divisor on $X$. We assume that  
the associated line bundle $L_D\geq 0$ is semi-positive, and the anti-canonical line bundle $-K_X$ can be decomposed into 
\[
-K_X = -(K_X + (1-\beta)L_D) +(1-\beta) L_D.
\] 
We further assume $-(K_X + (1-\beta)L_D) > 0$, and consider the cohomology class of 
\begin{align}\label{cohoclass}
\Om=-(K_X+(1-\beta)L_D).
\end{align}

Let $\mathcal{E}$ denote the space of all K\"ahler-Einstein cone metrics on $X$, with angle $2\pi\beta$ along the divisor $D$ and has $C^{2,\a}_\b$ K\"ahler cone potential. Assume that $\mathcal{E}$ is not empty, i.e. there exists a K\"ahler-Einstein cone metric $$\omega_\vphi= \omega+i\ddbar\vphi\in\mathcal{E},$$ with potential $\vphi\in C^{2,\alpha}_{\beta}$.
The background metric $\om$ is either a smooth K\"ahler metric $\om_0$ or the model metric $\om_D$.

Note that the K\"ahler cone potential of a K\"ahler-Einstein cone metric is $C^{1,1}_\b$, and indeed $C^{2,\a}_\b$ by the Evans-Krylov estimate of the K\"ahler-Einstein equation \eqref{keg} with Lemma \ref{f} (see Section \ref{Higher order estimate}).

We can choose $\phi_g$ as a metric (not a function!) of the $\bR$-line bundle $-(K_X+(1-\b)L_D)$ and write $$ \omega_\vphi = i\ddbar \phi_g. $$ The metric satisfies the following Monge-Amp\`ere equation: 
\begin{equation}\label{100}
(i\ddbar \phi_g)^n = e^{-\Phi},\ \ \ \ \int_X e^{-\Phi} = c_1(\omega)^n
\end{equation}
where $$\Phi = \phi_g + (1-\beta)\psi,$$ and $$\psi = \log |s|^2$$ is a positively curved singular hermitian metric (not a function!) on the line bundle $L_D$. Notice that the metric $\phi_{g}$ is in fact smooth on the regular part $M$, by applying the bootstrap method to the complex Monge-Amp\`ere equations. 

We furthermore discuss and write down the equivalent equations of \eqref{100}.
According to the cohomology condition, the metric $e^{-\Phi}$ is exactly a volume form. Hence equation \eqref{100} makes sense. 
Thanks to Poincar\'e-Lelong formula, we have 
$$i\ddbar\psi = 2\pi [D]. $$
Hence up to an normalization, equation \eqref{100} is equivalent to the following:
$$ -i\ddbar \log\omega_\vphi^n = i\ddbar\phi_g + 2\pi(1-\beta)[D]. $$
The two sides of this equation are globally defined, i.e. the equation which K\"ahler-Einstein cone metric satisfies.

Conversely, we are given a K\"ahler-Einstein cone metric which satisfies the equation of currents,
\begin{align}\label{ke}
Ric (\omega_{\vphi})=\om_\vphi+2\pi(1-\b) [D].
\end{align}
This equation implies the cohomology relation \eqref{cohoclass}. 
Using the smooth metric $\om_0$, we have the following equation from the cohomology relation \eqref{cohoclass},
\begin{equation}
\label{rew}
Ric(\omega_0) = \omega_0 + i\ddbar \Psi,
\end{equation}
where $\Psi/(1-\beta)$ is a smooth metric on the line bundle $L_D$. 
Put $$ \omega_0 = i\ddbar \phi_0,$$ and we have the following identity from equation (\ref{rew}).
$$ \exp(-\phi_{0} - \Psi) =  \omega_0^n,$$
then $$\phi_0=\phi_g -\varphi - \delta|s|^{2\beta}_h$$ is the metric for the K\"ahler form $\omega_0$.
Then combining \eqref{ke} and \eqref{rew}, we have 
\begin{equation}
\label{1002}
\frac{\omega_{\varphi}^n}{\omega_0^n} 
=e^{-\varphi -\delta|s|_h^{2\beta}+ \Psi-(1-\b)\psi}
= \frac{e^{-\varphi -\delta|s|_h^{2\beta}+ \Psi+(1-\b)\log h}}{|s|_h^{2-2\beta}}.
\end{equation}
Let $h_0$ be the smooth function $\Psi+(1-\b)\log h$.
In conclusion, under the smooth background metric $\om_0$, it becomes
\begin{align}\label{kegsmooth}
\frac{\omega_{\varphi}^n}{\omega_0^n} =\frac{e^{-\varphi -\delta|s|_h^{2\beta}+ h_0}}{|s|_h^{2-2\beta}}. \;
\end{align}
We denote $$\mathfrak f_0=-\log(|s|_h^{2(1-\b)})
-\delta|s|_h^{2\b}+h_0.$$
If we use $\om$ as the background metric, the K\"ahler-Einstein cone metric $\om_\vphi=\om+i\p\bar\p\vphi$ satisfies
\begin{align}\label{keg}
 \log \omega_{\vphi}^{n} = \log \omega^{n} -\vphi+\mathfrak f. \;
\end{align}
Here
\begin{align}\label{spgeneral}
\mathfrak f=\mathfrak f_0-\log\frac{\om^n}{\om^n_0}=-\log(\frac{\om^n}{\om^n_0}|s|_h^{2(1-\b)})
-\delta|s|_h^{2\b}+h_0.
\end{align}
In particular, one could choose the background K\"ahler cone metric to be the model metric $\om=\om_D$.
The estimates of $f$ defined by $\om_D$ are useful in the higher order estimates (see Lemma 4.1 in Calamai-Zheng \cite{Calamai-Zheng}). 
\begin{lem}\label{f}
$\mathfrak f\in C^{\a}_\b$ for any $0<\b\leq 1$ and $\a\leq\min\{\frac{2}{\b}-2,\frac{1}{\b} \}$.
$|\p\mathfrak f|_{\om_D}$ is bounded when $0<\b<\frac{2}{3}$.
\end{lem}

\begin{rem}
The lemmas above follow for all normal crossing divisors $D$.
\end{rem}

\section{The automorphism group is reductive}\label{sec01}

Now let's call $\text{Aut}(X,D)$ as the set of all holomorphic automorphisms of $X$, which fix the divisor $D$.
And assume $G$ is the identity component of $\text{Aut}(X,D)$. Let $\mathfrak{g}$ be the space of all holomorphic vector fields on $X$ tangential to $D$. 
Fix a $C^{2,\alpha}_{\beta}$ cone metric $\theta$, and then 
we can consider its isotropy group $K_{\theta}$ of $G$. The $G$-orbit $\mathcal{O}$ through $\theta$ in $\mathcal{E}$ can be written as
$$\mathcal{O} \cong G/K_{\theta}.$$
Take $\mathfrak{k}_{\theta}$ to be the set of all Killing vector fields on $X$ with respect to $\theta$, and $\mathfrak{k}_{\theta}$ is the Lie sub-algebra of $\mathfrak{g}$
corresponding to $K_{\theta}$ in $G$. Our goal is to prove the following:

\begin{theorem}\label{reductivityStatement}
\label{a}
Let $$H_{\theta}: = \{\varphi\in  C^{2,\alpha}_{\beta} \cap C^{\infty}(M) | \ \Delta_{\theta} \varphi = -\varphi \},$$ 
where $\Delta_{\theta}$ is the geometric Laplacian, and $M$ is the complement of $D$ on $X$.
Set $\mathfrak{p}_{\theta}: =\sqrt{-1}\mathfrak{k}_{\theta}$, and 
$H_{\theta}^{\mathbb{C}}: = H_{\theta}\otimes_{\mathbb{R}} \mathbb{C}$. Then 
\begin{enumerate}
\item[(i)] $\mathfrak{k}_{\theta} = \{ Y_{\varphi,\theta}| \varphi\in \sqrt{-1} H_{\theta} \}$ and $\mathfrak{p}_{\theta} = \{Y_{\varphi, \theta}| \varphi\in H_{\theta} \}$.\\

\item[(ii)] $\varphi\in H_{\theta}^{\mathbb{C}}\rightarrow Y_{\varphi,\theta} \in \mathfrak{g}$ defines an isomorphism and hence 
$$\mathfrak{g} = \mathfrak{k}_{\theta}\oplus \mathfrak{p}_{\theta} . $$
\end{enumerate}
\end{theorem}

In order to prove above theorem, it is enough to prove the following two statements: first, given a holomorphic vector field $v$ tangential to $D$, 
we can create a corresponding element $$u\in H^{\mathbb{C}}_{\theta};$$ second, given an element $u_2\in H_{\theta}$, we can induce a holomorphic 
vector field $v_2\in\mathfrak{g}$ from $u_2$. We will prove the first statement by solving a $\dbar$ equation, and the second statement is proved by applying 
a Bochner-Kodaira type formula.

\subsection{Solving $\dbar$ equation}
We clarify our notations again.
Let $$\omega_g:=\omega_{\vphi_g}=\omega + i\p\bar\p\varphi_g,$$ be a K\"ahler-Einstein cone metric with angle $2\pi\beta$ along $D$, with potential $\vphi_g$ in $ C^{2,\alpha}_{\beta}$. 

Suppose 
$v$ is a holomorphic vector field on $X$ in $T^{0,1}(X)$, or equivalently, a holomorphic $(n-1,0)$ form with value in $-K_X$. We define $(n,1)$-form with value in $-K_X$ as $$f:=\omega_{g}\wedge v,$$ and consider the equation:
\begin{equation}\label{101}
f = \dbar u.
\end{equation}
In general, it's not easy to handle equation (\ref{101}), even in the $L^2$ sense. However, we have the following proposition when $v$ is tangential to the divisor.
First we claim that $f$ is a closed $(0,1)$-current on $X$.
\begin{lemma}
\label{lem-1}
The $(0,1)$-current $f=\omega_g\wedge v$ is $\dbar$ closed.
\end{lemma}
\begin{proof}
It's enough to check the following: let $U$ be an open neighborhood around a point $p\in D$, 
for any smooth $(0,n-2)$ form $W$ such that $\text{supp} W \Subset U$, we have 
$$\int_X \omega_g\wedge v\wedge \dbar W =0.  $$
The convolution $\varphi_{g,\ep} = \chi_{\ep} \star \varphi_g $ converges uniformly to $\varphi_g$ locally. Hence we have weak convergence as 
$$\int_X \omega_{g, \ep}\wedge (v\wedge \dbar W)  \rightarrow \int_X \omega_g \wedge (v\wedge \dbar W). $$ By integration by parts, we have 
$$ \int_X \dbar(\omega_{g,\ep} \wedge v)\wedge W = 0,$$ for each $\ep$. And the result follows.
 
\end{proof}

We denote $W^{1,2}(\om_0)$ the $W^{1,2}$ Sobolev space with respect to the smooth K\"ahler metric $\om_0$.

\begin{proposition}\label{dbar}
Suppose the holomorphic vector field $v$ is tangential along $D$. Then there exists a function $u\in C^{\infty}(M)\cap W^{1,2}(\om_0)$, such that $u$ solves equation (\ref{101}), and the following estimate holds: 
\begin{equation}\label{102}
\int_X |u|^2 e^{-\Phi} \leq \int_X H e^{-\Phi},
\end{equation}
where $H = |f|_{i\ddbar\phi_g}$ is the $L^2$ norm of $f$ under the metric $\omega_g$. 
\end{proposition} 
\begin{proof}
We can write $v = X^i d\hat{z}^i$ locally, where $d\hat{z}^i$ is an $(n-1,0)$ form defined by $$dz^i\wedge d\hat{z}^i := dz^1\wedge\cdots\wedge dz^n=dZ,$$ and $X^i$ is a holomorphic function with value in $-K_X$. Then $$\omega_g\wedge v = (X^{\alpha}g_{\alpha\ol\beta}) d\bar{z}^{\beta}\wedge dZ$$ is an 
$(n,1)$ form with value in $-K_X$ (note that those coefficients may differ by a sign, but we ignore this problem here since we only concern about $L^p$ norms).

Notice that, from Lemma \ref{lem-1}, $\dbar f =0$ on $X$ shows that $$f=\omega_g\wedge v$$ is a $\dbar$ closed $(0,1)$ form, and $X$ is in fact a projective manifold by the ampleness of $-K_X$. Then the result follows from a slightly general version \cite{Blocki} of H\"ormander's $L^2$ estimate \cite{Bo}, and it's enough to check two things: $f$ is in $L^2_{loc,(0,1)}$ and 
$H\in L^{\infty}_{loc}$ satisfies
\begin{equation}\label{1001}
f\wedge\ol f \leq H i\ddbar \Phi, 
\end{equation}
 in the sense of currents of order zero (measure coefficients).

These conditions are true thanks to the vanishing of the orthogonal direction of $v$ near the divisor. In fact, we can decompose $X^1 = s\cdot h$ near the divisor $D$, where $D=\{s=0\}$ and $h$ is a local holomorphic function. Then we can check the growth order of $f$ near $D$ as:
\begin{eqnarray}\label{103}
f_{\ol1} &=& X^1g_{1\ol1} + \sum_{j>1}X^j g_{j\ol1}\nonumber\\
&\sim& r^{2\beta -1} + r^{\beta -1},
\end{eqnarray}
and for $k>1$ 
\begin{eqnarray}\label{104}
f_{\ol k} &=& X^1g_{1\ol k} + \sum_{j>1}X^j g_{j\ol k}\nonumber\\
&\sim&  r^{\beta },
\end{eqnarray}
where $|z^1|=r$. Hence we have $f\in L^2_{loc}$ and $H\in L^{\infty}_{loc}$, since $$|f|^2_{\omega_g} \leq C |f|^2_{\omega_{D}},$$
where $\omega_{D}$ is the model cone K\"ahler metric, and the latter is bounded since $$r^{2-2\beta}|f_{\ol 1}|^2\sim O(1)\text{ and 
}|f_{\ol k}|^2\sim O(1).$$ Finally notice that $i\ddbar\Phi$ can be written as $$i\ddbar\Phi = i\ddbar \phi_g + (1-\beta)\delta_D,$$ where $\delta_D$ is the integration 
current of $D$. Therefore we can establish the inequality:
\begin{equation}\label{105}
f\wedge\ol f \leq H i\ddbar{\phi_g} = H i\ddbar\Phi,
\end{equation}
on $M$ by definition of $H$. However, the coefficients of $f\wedge \ol f$ has no mass on the divisor $D$ since $f$ is $L^2_{loc}$. Hence inequality (\ref{104}) actually holds on the whole $X$.  

\end{proof}

\begin{rem}\label{hormander}
In fact, we can solve the $\dbar$ equation (\ref{101}) with estimate (\ref{102}) under even weaker conditions, provided that inequality (\ref{1001}) still holds in the sense of complex measure coefficients positive $(1,1)$ currents, and the integral on the RHS of equation (\ref{102}) is finite. 
\end{rem}

Next let's consider the complex Laplacian operator $\Box_g$ defined with respect to the K\"ahler-Einstein cone metric $\omega_g$. It can be written as 

$$ \Box_g u :=  -g^{\ol\beta\alpha}\frac{\partial^2 u}{\partial z^{\alpha}\partial\ol z^{\beta}}=-\tri_{\om_g}, $$
 in a local coordinate system. It certainly makes sense to define it outside of the divisor $D$, and it also makes sense across the divisor when $u$ is merely in $ C^{2,\alpha}_{\beta}$.

 Now we can look at this operator in a different view of point. We are given a $C^{\a}_\b$ K\"ahler cone metric $\om$.  We say a form  $f$ in $L^2(\om,\Phi)$, if $$\int_X |f|^2_\om\cdot e^{-\Phi} < +\infty.$$

Define $\dbar$ operator as a closed, densely defined operator between two Hilbert spaces, with closed range property. That is to say $$\dbar: L^2_{(n,0)}(\omega, \Phi) \dashrightarrow L^2_{(n,1)}(\omega, \Phi), $$ where $\Phi$ is viewed as a positively curved singular Hermitian metric on the anti-canonical line bundle $-K_X$. Then there exists its adjoint operator 

$$\dbar^*_{\Phi,\omega}: L^2_{(n,1)}(\omega, \Phi) \dashrightarrow L^2_{(n,0)}(\omega, \Phi),$$ 
which is also a closed, densely defined operator with closed range. However, there is another way to define the formal adjoint operator of $\dbar$, by doing integration by parts in local coordinate systems.

It can be written as, for any $-K_X$ valued $(n,1)$ form $f$,
$$ \vartheta f = \partial^{\Phi}(\omega\lrcorner f ), $$ 
in the distributional sense, and the operator $\partial^{\Phi}$ is defined as 
\begin{equation}\label{dphi}
\partial^{\Phi}\cdot := e^{\Phi}(\partial e^{-\Phi}\cdot) = \partial - \partial\Phi\wedge\cdot.
\end{equation}
It's standard to show $\dbar^*_{\Phi,\omega} = \vartheta$ on the domain of $\dbar^*_{\Phi,\omega}$. Therefore we can abuse them and define the other second order elliptic operator as 
\begin{equation}\label{106}
\Box_{\Phi,\omega} u := \dbar^*_{\Phi,\omega}\dbar u = \partial^{\Phi}(\omega\lrcorner \dbar u). 
\end{equation}

If we put the metric $\omega = \omega_g$, then a quick observation \cite{Li1} is that 
these two operators $\Box_{g}$ and $\Box_{\Phi,\omega_g}$ coincides with each other on $M$. 
Hence we can translate the Laplacian equation 
 into two first order equations:
\begin{equation}\label{107}
\left\{ \begin{array}{rcl}
\omega_g\wedge v &=& \dbar u\\
\partial^{\Phi}v &=& \Box_g u,
\end{array}\right.
\end{equation}
where in prior, $v$ is a vector field on $M$. 

However, the operator $\Box_{\Phi,\omega_g}$ is not quite well defined as a global operator, since it's not clear that $\dbar$ operator has closed range in the $L^2$ space with singular metric $\omega_g$ (it's proved by the Bochner technique, which involves one derivative of the metric $\omega_g$). The key observation here is that the operator $\partial^{\Phi}$, defined in equation (\ref{dphi}), is independent of the metric $\omega$. 
Then it still makes sense to talk about the system of differential equations like (\ref{107}) on the whole manifold $X$ in the current sense, and we are going to consider it in a very special circumstance. 
\begin{lemma}\label{2eq}
Under the same conditions in Proposition \ref{dbar}, the following equation holds on $X$:
\begin{equation}\label{1075}
\partial^{\Phi}v = u + C,
\end{equation}
where $C$ is some normalization constant. In particular, the function $u$ is in $C^{\bar 1,\alpha}_{\beta}$, i.e. $u\in  C^{\alpha}_{\beta}$ and $\dbar u\in  C^{\alpha}_{\beta}$. 
\end{lemma}
\begin{proof}
First note that outside the divisor $D$, we can write equation (\ref{101}) as $$\ddbar \Phi \wedge v = \dbar u.$$ Then by the commutation relation 
$ \partial^{\Phi}\dbar + \dbar\partial^{\Phi} = \ddbar\Phi, $ we derive the following $\dbar$ equation on $M$:
\begin{equation}\label{108}
\dbar(\partial^{\Phi}v - u) =0.
\end{equation}
The difference $\xi = \partial^{\Phi}v - u$ is a holomorphic function outside the divisor. Then a standard theorem (Lemma 1.1, Lecture 5 \cite{Bo}) implies that $\xi$ can be extended across the divisor $D$, provided 
$\partial^{\Phi}v$ and $u$ are in $L^2_{loc}$. The norm $||u||_{L^2}$ is bounded thanks to the $L^2$ estimate (\ref{102}), and notice that we can compute $\partial\Phi$ on $M$ as:
$$\partial\Phi = \partial \phi_g + (1-\beta)\frac{\partial s}{s}.$$
But $ h=\frac{X^1}{z^1}$ is a local holomorphic function near a point on the divisor. Hence the following equation holds on all of $X$: 
\begin{equation}\label{109}
\partial\Phi\wedge v = \partial\phi_g\wedge v + (1-\beta)\frac{\partial s}{s}\wedge v. 
\end{equation}
Now we can write $$\partial^{\Phi}v = F - \partial\phi_g \wedge v,$$ where $F$ is a holomorphic function. In particular, $\partial^{\Phi}v$ is in $L^2$, and we even have a better regularity. The singular term can be decomposed as follows:
\begin{equation}\label{110}
\partial\phi_g\wedge v = (X^1 \frac{\partial\phi_g}{\partial z^1} + \sum_{j>1} X^j \frac{\partial\phi_g}{\partial z^j}) dZ. 
\end{equation}
The sum on the RHS of above equation is a smooth function, and the first term has the following growth control near the divisor:
\begin{equation}\label{111}
X^1 \frac{\partial\phi_g}{\partial z^1} \sim r^{\beta}o(1);\ \ \ \ \dbar(X^1 \frac{\partial\phi_g}{\partial z^1})=
 X^1\dbar \big( \frac{\partial\phi_g}{\partial z^1}\big) \sim r^{2\beta-1}O(1). 
\end{equation}
Hence, the coefficients of $\partial^{\Phi}v$ is in $ C^{\alpha}_{\beta}$ and the coefficients of $\dbar\partial^{\Phi}v$ is in $C^{\alpha}_{\beta}$. 
Finally, this shows the difference $\xi$ is a global holomorphic function on $X$, which can only be a constant. 
\end{proof}

Next we claim that the function $u$ constructed in Proposition \ref{dbar} is in the eigenspace $\Lambda_1$ of the Laplacian operator $\Delta_g$ with eigenvalue $1$ (the smallest eigenvalue). To see this, we first need a normalization condition:
\begin{equation}\label{111}
\int_X u \cdot e^{-\Phi} = 0.
\end{equation}
There are two ways to look at this equation: first, $u$ is a $-K_X$ valued $(n,0)$ form, which is exactly a function on $X$, and $e^{-\Phi}$ is a volume form, so the integral makes sense; second, it is equivalent to write equation (\ref{111}) as 
$$\int_X u\wedge\ol U e^{-\Phi}=0, $$ where $U$ is a $-K_X$ valued $(n,0)$ form, which is the representative of the constant function $1$ on $X$.
Then $e^{-\Phi}$ is viewed as the metric on the anti-canonical line bundle $-K_X$, and equation (\ref{111}) really says that $u$ is orthogonal to the kernel of $\dbar$ operator under the weight $e^{-\Phi}$. Based on this normalization, we have the following lemma:
\begin{lemma}\label{C}
Under the same conditions in Proposition \ref{dbar}. If we normalize the function $u$ as equation (\ref{111}), i.e. $u\perp_{\Phi} \ker{\dbar}$, then the constant $C$ appearing in Lemma \ref{2eq} is zero.
\end{lemma}
\begin{proof}
It's enough to prove the following identity:
\begin{equation}\label{112}
\int_X \partial (ve^{-\Phi}) = \int_X (\partial v - \partial\Phi\wedge v) e^{-\Phi} = 0. 
\end{equation}
Let's first consider a smooth approximation sequence of $\Phi$: $$ \Phi_{\ep} = \phi_g + (1-\beta)\log (|s|^2 + \ep e^{\psi}), $$ 
where $\psi$ is a smooth positively curved metric on the line bundle $L_D$. Then we know $\Phi_{\ep}$ is decreasing to $\Phi$, and $i\ddbar \Phi_{\ep} \geq \omega_g$ \cite{Li2}. 
Now it's trivial to see $$\int_X \partial(ve^{-\Phi_{\ep}}) = 0.$$ Then we claim the integrals will converge to $\int_X\partial{v\cdot e^{-\Phi}}$. Notice that we can write the integral as 
\begin{equation}\label{114}
\int_X \partial (ve^{-\Phi_{\ep}}) = \int_X \partial v\cdot e^{-\Phi_{\ep}} - \int_X \partial\phi_g\wedge v e^{-\Phi_{\ep}} 
- \int_X \partial \log (|s|^2 + \ep e^{\psi}) \wedge v e^{-\Phi_{\ep}}.
\end{equation}
The first two terms on the RHS of above equation will converges to $$\int_X (\partial v - \partial\phi_g\wedge v)\cdot e^{-\Phi},$$ by Lebesgue's dominated convergence theorem. 
But the third term is the tricky part here. 
\begin{eqnarray}\label{115}
\partial \log (|s|^2 + \ep e^{\psi}) \wedge v e^{-\Phi_{\ep}} &=& \frac{\ol s \partial s + \ep \partial\psi e^{\psi}}{(|s|^2+\ep e^{\psi})}\wedge v e^{-\Phi_{\ep}}
\nonumber\\
&=& \frac{\ol s \partial s\wedge ve^{-\phi_g}}{(|s|^2+\ep e^{\psi})^{2-\beta}} + \frac{\ep e^{\psi}\partial\psi\wedge ve^{-\phi_g} }{(|s|^2+\ep e^{\psi})^{2-\beta}}.
\end{eqnarray}
The first term in the last line of equation (\ref{115}) is safe since $$\partial s\wedge v = s\cdot h$$ for some holomorphic function $h$ locally near the divisor. 
For the second term, it's enough to estimate it locally in the orthogonal direction to the divisor $D$. For $z^1\in \bC$, we can compute the following:
\begin{eqnarray}\label{116}
\int_{|z^1|<1} \frac{\ep dz^1\wedge d\bar{\zeta}}{(|z^1|^2+\ep)^{2-\beta}} &=& c\ep \int_0^1  \frac{rdr}{(r^2+\ep)^{2-\beta}}
\nonumber\\
&=& c\ep (\ep^{\beta-1} + O(1))\nonumber\\
&\sim& \ep^{\beta},
\end{eqnarray}
where $r = |z^1|$ and $c$ is some uniform constant. Hence the second term converges to zero when $\ep\rightarrow 0$, which implies the convergence of the integral, i.e. $$\lim_{\ep\rightarrow 0}\int_X \partial(v\cdot e^{-\Phi_{\ep}}) = \int_X \partial(v\cdot e^{-\Phi}) =0.$$
\end{proof}
\begin{rem}
It's easy to see that equation (\ref{112}) holds locally near the divisor, by considering this integration on a sequence of subdomains defined as $D_{\ep} = \{|s|>\ep \}$.
However, this integration by parts argument can not be directly applied to our situation. This is because, 
on the one hand, the defining function $|s|$ is not well defined globally; on the other hand, $\nabla |s|_h$ will generate non-parallel directions to the tangential direction of the divisor.
\end{rem}

Now if we combine Lemma \ref{2eq} and Lemma \ref{C}, then outside the divisor $D$, the function $u$ satisfies 
$$\Box_g u = u.$$
That is to say, the function $u$ is in fact an eigenfunction of $\Box_g$ with smallest eigenvalue $1$ outside the divisor. 
\begin{lemma}\label{1110}
Let $\omega_g$ be a K\"ahler-Einstein cone metric with angle $2\pi\beta$ along a smooth divisor $D$. Suppose $u\in C^{\alpha}_{\beta}$ is a function such that the following things hold:
$$\Box_{\omega_g} u = u,$$ on $M$, 
and 
\begin{equation}\label{L2}
\int_X|u|^2 e^{-\Phi} < C, 
\end{equation}
with
\begin{equation}\label{L21}
\int_X |\nabla u|^2_{\omega_g} e^{-\Phi} < C, 
\end{equation}
where the norm for the $(0,1)$ form is taken with respect to the cone metric $\omega_g$.
Then $u$ is in $ C^{2,\alpha}_{\beta}$. 
\end{lemma}
\begin{proof}
We will only sketch the proof here. From \eqref{L2} and \eqref{L21}, Section 5 in \cite{Calamai-Zheng} implies that $u$ is a $W^{1,2}$ weak solution. Then the Harnack inequality, Proposition 5.12 proved in \cite{Calamai-Zheng}, implies that $u$ has bounded $C^{\a}$ norm. Thus the conclusion follows from applying Donaldson's $C^{2,\alpha}_{\beta}$ Schauder estimate to the equation $\Box_{\omega_g} u = u$.

\end{proof}

Observe that inequality (\ref{L2}) is equivalent to say $u\in L^2(\Phi)$, which is guaranteed by the H\"ormander's estimate (Proposition \ref{dbar}). 
Moreover, the condition $\nabla u \in L^2(\omega_g,\Phi)$ is also true by the following lemma.
\begin{lemma}\label{1111}
The function $u$ constructed in Proposition \ref{dbar} satisfies 
\begin{equation}\label{W2}
\int_X |\dbar u|^2_{\omega_g}e^{-\Phi} < +\infty,
\end{equation}
and 
\begin{equation}\label{W21}
\int_X |\partial u|^2_{\omega_g}e^{-\Phi} < +\infty.
\end{equation}
In particular, $u\in C^{2,\alpha}_{\beta}$.
\end{lemma}
\begin{proof}
First observe that for any $(0,1)$ form $\alpha$, the two norms $|\alpha|^2_{\omega_g}$ and $|\alpha|^2_{\omega_{D}}$ are equivalent locally near a point on the divisor, where $\omega_{D}$ is the standard model cone metric, by the isometric property between these two metrics. Now we have seen in the proof of Proposition \ref{dbar} that $$|u_{,\ol1}|\sim r^{\beta -1}\text{ and }|u_{,\ol k}|\sim r^{\beta}\text{ for }k>1.$$ Then we have
\begin{equation}\label{comp}
|\dbar u|^2_{\omega_{D}} = r^{2-2\beta}|u_{,\ol1}|^2 + \sum_{k>1}|u_{,\ol k}|^2\sim O(1).
\end{equation}
Hence $\dbar u\in L^2(\omega_g, \Phi)$. Next noticed that those derivatives on the tangential directions are all in $L^{2}(\omega_g, \Phi)$. This is because locally we can write for all $k>1$
\begin{equation}\label{ibp}
\int \left( \frac{\partial u}{\partial \ol z^k}  \right) \ol{\left( \frac{\partial u}{\partial \ol z^k}  \right)}\frac{1}{|z^1|^{2-2\beta}}
= \int \left( \frac{\partial u}{\partial z^k}  \right) \ol{\left( \frac{\partial u}{\partial z^k}  \right)}\frac{1}{|z^1|^{2-2\beta}},
\end{equation}  
by Fubini's theorem and a convolution argument (compare to Theorem 4.2.5, H\"ormander \cite{Hor}). Then the lemma will follow if we can prove 
$|u_{,1}|\in L^2$, since $$|\partial u |^2_{\omega_{D}} = r^{2-2\beta}|u_{,1}|^2 + \sum_{k>1}|u_{, k}|^2.$$
But this is true since $u\in W^{1,2}(\om_0)$. 
\end{proof}

All in all, we conclude as follows. 
\begin{theorem}\label{main1}
Suppose there exists a holomorphic vector field $v$ tangential to the divisor $D$. Then the function $u\in  C^{2,\alpha}_{\beta}$ constructed in Proposition \ref{dbar} satisfies the following equation on $X$(interpreted as the linear system (\ref{107})):
\begin{equation}\label{117}
\Box_{g} u = u.
\end{equation}
In particular, $u$ is in the eigenspace $\Lambda_1$ of the Laplacian operator $\Delta_g$ with eigenvalue $1$. 
\end{theorem}
\subsection{Creating the holomorphic vector field}\label{sub2}
The remaining task is to prove a theorem ``going backwards". That is to say, to create a holomorphic vector field from a real valued egienfunction $u_2$ in the eigenspace $\Lambda_1$. More precisely, when $u_2$ is chosen as the imaginary part of the function $u\in\Lambda_1$, 
we want to prove the induced vector field $\uparrow\dbar u_2$ is holomorphic. Then its real part is a Killing vector field, and 
this implies 
the automorphism group is the complexification of the group of Killing vector fields, i.e. $\Aut (X,D) = K^{\bC}$. Then it is reductive. 

For any $u\in\Lambda_1$, let's write $u=u_1 + \sqrt{-1}u_2$, where $u_1$ and $u_2$ are real valued functions. We see $u_1$ and $u_2$ also satisfy equation (\ref{117}) on $M$, since
the Laplacian operator $\Box_g$ is a real operator for the K\"ahler-Einstein cone metric $\omega_g$. Then the following system of differential equations holds for the function $u_2$ on $M$:
\begin{equation}\label{118}
\left\{ \begin{array}{rcl}
\omega_g\wedge v_2 &=& \dbar u_2\\
\partial^{\Phi}v_2 &=&  u_2,
\end{array}\right.
\end{equation}
Now we want prove the following theorem.
\begin{theorem}\label{main3}
The vector field $v_2$ is a holomorphic vector field tangential to the divisor $D$. 
\end{theorem}

First notice that $v_2$ has $L^2$ coefficients. This is because $\dbar u_2\in L^2$, and locally in a normal coordinate around an arbitrary point $p\in M$, we have
\begin{eqnarray}\label{120}
|v_2|^2(p) &=&  h_{\alpha\ol\mu}g^{\ol\beta\alpha}u_{2,\ol\beta} g^{\ol\mu\gamma}u_{2,\gamma} 
\nonumber\\
&=& \sum_{\alpha} \frac{1}{\lambda^2_{\alpha}} |u_{,\ol{\alpha}}|^2
\nonumber\\
&\leq& c^{-2} \sum_{\alpha}   |u_{,\ol{\alpha}}|^2,
\end{eqnarray}
where we used the inequality $\omega_g \geq c \omega$. Then observe that $\partial^{\Phi}v_{2}\in L^2$ by the second equation of ({\ref{118}}). 
In fact, we can gain more regularities of $v_2$ from $u$ as follows

\begin{lemma}\label{reg}
$u_2\in  C^{2,\alpha}_{\beta}$. In particular, $v_2\in L^2(\omega_g,\Phi)$ and $\partial^{\Phi}v\in L^2(\Phi)$. 
\end{lemma}  
\begin{proof}
By Lemma \ref{1111}, $u\in L^2(\Phi)$ and $\nabla u\in L^2(\omega_g, \Phi)$, which implies $u_2\in L^2(\Phi)$ and $\nabla u_2\in L^2(\omega_g, \Phi)$. 
Hence $u_2\in  C^{2,\alpha}_{\beta}$ by Lemma \ref{1110}. 
\end{proof}
However, the true obstruction is that we don't know the growth of $\dbar v_2$ (even $L^2$ is unclear!) near the divisor, where the third derivatives of the potential are involved. 

\subsection{Cut-off function}
In order to circumvent this problem, we need to invoke a useful cut off function (Lemma 2.2 \cite{Bo11}).
First let $$\eta: \bR^{+}\rightarrow \bR^{+}\cup \{0 \} $$ be an auxiliary function, which is a non-decreasing smooth function such that $\eta =0$ 
when $x< 1$ and $\eta = 1$ for $x>2$ with $|\eta'|$ and $|\eta''|$ bounded. Then define for any $\ep>0$ small, 
\begin{equation}\label{121}
\rho_{\ep}:= \eta(\ep\log(-\log |s|^2_h)), 
\end{equation}
where $h=e^{-\psi}$, a smooth positively curved hermitian metric on the line bundle $L_D$, and we can always normalize $|s|^2_h < 1$ on $X$.  For the convenience of readers, we compute its derivatives as follows.
\begin{lemma}\label{Bo}
Let $\tau = |s|^2e^{-\psi}$ be the $L^2$ norm of the section. On $M$, we can write 
\begin{equation}\label{122}
\dbar\rho_{\ep} =  \frac{\ep\eta'}{\log\tau} \ol{\left( \frac{\partial^{\psi}s}{s}\right)}, 
\end{equation}
and 
\begin{equation}\label{123}
\partial\dbar\rho_{\ep} = -\ep\eta' \frac{\partial\dbar\psi}{\log\tau} + (\ep^2\eta'' - \ep\eta') \left | \frac{\partial^{\psi}s}{s\log\tau} \right |^2.
\end{equation}
In particular, $||\dbar\rho_{\ep}||_{L^2(\om_0)} \rightarrow 0$ as $\ep\rightarrow 0$. 
\end{lemma}

\begin{proof}
Let $K =\log\tau$, and it derivative is 
\begin{equation}\label{124}
\partial K = \frac{\ol s\partial ( { s e^{-\psi}}) }{|s|^2e^{-\psi}} = \frac{\partial^{\psi} s}{s}.
\end{equation}
The function $\rho_{\ep}$ can be written as 
$$\rho_{\ep} = \eta(\ep\log(-\log\tau)). $$
Hence take $\dbar$, we have 
\[
\dbar \rho_{\ep} = \eta' \ep \frac{\dbar K}{K} = \frac{\ep\eta'}{K}\ol{\left ( \frac{\partial^{\psi}s}{s} \right )},
\]
which proved equation (\ref{122}). Take $\partial$ again, we have 
\begin{eqnarray}\label{125}
\ddbar\rho_{\ep} &=& \partial (\ep\eta' \frac{1}{K} \ol{\left ( \frac{\partial^{\psi}s}{s} \right )})
\nonumber\\
&=& \eta'' \ep^2 \frac{1}{K^2}\left ( \frac{\partial^{\psi}s}{s} \right )\wedge \ol{\left ( \frac{\partial^{\psi}s}{s} \right )} 
+ \eta'\ep \partial \left( \frac{1}{K} \ol{\left ( \frac{\partial^{\psi}s}{s} \right )}\right).
\end{eqnarray}
Compute the last term as 
\begin{eqnarray}\label{126}
 \partial \left( \frac{1}{K} \ol{\left ( \frac{\partial^{\psi}s}{s} \right )}\right) &=&
-\frac{1}{K^2} \frac{\partial^{\psi}s}{s}\wedge \ol{\left ( \frac{\partial^{\psi}s}{s} \right )} + \frac{1}{K}\ol{\left( \frac{\dbar\partial^{\psi}s}{s} \right)}
\nonumber\\
&=& -\frac{1}{K^2} \frac{\partial^{\psi}s}{s}\wedge \ol{\left ( \frac{\partial^{\psi}s}{s} \right )} + 
\frac{1}{K} \ol{\ddbar\psi},
\end{eqnarray}
where we used the commutation relation $$\ddbar\psi = \dbar\partial^{\psi} + \partial^{\psi}\dbar$$ in the last equation. Combine equations (\ref{125}) and (\ref{126}), we proved equation (\ref{123}). And the convergence follows easily, since locally on the orthogonal direction, 
$$\partial\rho_{\ep}\wedge \dbar\rho_{\ep} \sim \ep^2 \omega_P,$$
where $\omega_P$ stands for the Poinc\'are metric on the unit disk, which always has a finite volume. 
\end{proof}

The cut off function $\rho_{\ep}$ is supported on a small neighborhood 
$$D_{\ep} = \{ |s|^2_h < \exp(-e^{\frac{1}{\ep}}) \}$$ of the divisor, equals to $1$ on $D_{\ep/2}$. Of course the support converges to the divisor when $\ep\rightarrow 0$.

Before using these cut off functions to construct an approximation, let's first assume that there is sequence of smooth vector fields $v_{\ep}$, such that they belong to the following family.
\begin{equation}
\mathcal{V}_{\ep}: = \{ v\ ; \dbar v\wedge \omega_g =0 \ and\ v = 0\ on\ D_{\ep} \}.
\end{equation}
Then we have the following integration by parts formula for such vector fields.

\begin{lemma}\label{Bochner}
If $v_{\ep}\in \mathcal{V}_{\ep}$, then 
\begin{equation}\label{130}
 \int_X | \dbar v_{\ep}|^2_{\omega_g}e^{-\Phi} =  \int_X | \partial^{\Phi}v_{\ep}|^2e^{-\Phi} 
 - \int_X \omega_g\wedge v_{\ep}\wedge\ol v_{\ep} e^{-\Phi}. 
\end{equation} 
\end{lemma}
\begin{proof}
The observation is that $\Phi$ or $\partial\Phi$ only have singularities along the divisor $D$. Hence integration by parts works for free, provided one of the integrand is identically zero in a neighborhood of $D$. Then we compute as follows:
\begin{eqnarray}\label{IBP}
- \int_X \dbar v_{\ep}\wedge \ol{\dbar v_{\ep}}e^{-\Phi} &=& \int_X v_{\ep}\wedge\ol{\partial^{\Phi}\dbar v_{\ep}}e^{-\Phi}
\nonumber\\
&=& \int_X v_{\ep}\wedge \ol{( \omega_g\wedge v_{\ep} - \dbar\partial^{\Phi}v_{\ep} )}e^{-\Phi}
\nonumber\\
&=& \int_X \partial^{\Phi}v_{\ep}\wedge\ol{\partial^{\Phi}v_{\ep}}e^{-\Phi}  - \int_X \omega_g\wedge v_{\ep}\wedge\ol v_{\ep} e^{-\Phi}.
\end{eqnarray}
The first line holds because $\dbar v_{\ep}$ is zero on $D_{\ep}$, and the last line is because $v_{\ep}$ vanishes on $D_{\ep}$. Then by the assumption, $\dbar v_{\ep}$ is primitive with respect to the metric $\omega_g$, which implies
$$|\dbar v_{\ep}|^2_{\omega_g} = - \dbar v_{\ep}\wedge \ol{\dbar v_{\ep}}. $$
\end{proof}

Now if we put 
$\chi_{\ep} = 1-\rho_{\ep}$, then there are two nature ways of approximating:
$$ u_{\ep} = \chi_{\ep}u_2,$$
or $$ v_{\ep} = \chi_{\ep}v_{2}. $$
Let's look at the first approximation $u_{\ep} = (1-\rho_{\ep})u_2$, and we can define $$w_{\ep} = \uparrow^{\omega_g}\dbar u_{\ep}.$$ Then $w_{\ep}$ is indeed in $\mathcal{V}_{\ep}$, and $$\partial^{\Phi}w_{\ep} = \Box_{g}u_{\ep}.$$ Hence Lemma \ref{Bochner} implies 
\begin{equation}\label{op}
\int_X | \dbar w_{\ep}|^2_{\omega_g}e^{-\Phi} = \int_X ( | \Box_g u_{\ep}|^2 - |\dbar u_{\ep} |^2) e^{-\Phi}.  
\end{equation}
However, the growth of the Laplacian of the cut off function $\rho_{\ep}$ is too fast near the divisor. 
($\Delta_{g}\rho_{\ep} \sim \ep r^{-2\beta}(\log r)^{-2} ,$ which is in $L^2$ only when $\beta<1/2$ and never in $L^2(\Phi)$!).
From now on, we assume $$v_{\ep} = \chi_{\ep}v_2 \in \mathcal{V}_{\ep}.$$ Then let's invoke the following Bochner type identity for $(n,q)$ forms with value in certain line bundle $L$, which goes back to Siu, and reformulated by Berndtsson \cite{Bo}. 
Recall that $\om_0$ is smooth K\"ahler metric.
\begin{defn}
Let $\alpha$, $\beta$ be two differential forms with bidegree $(n,q)$ with value in a line bundle $L$. Then
\begin{equation}\label{T}
T_{\alpha}: = c_{n-1} \gamma_{\alpha}\wedge \ol\gamma_{\alpha}\wedge\omega_0^{q-1} e^{-\phi}, 
\end{equation}
where $c_{n-1}= i^{(n-q)^2}$ is a constant to make $T_{\alpha} \geq 0$, and $\gamma_{\alpha}$ is the unique $(n-q,0)$ form 
associated to $\alpha$ such that 
\[
\gamma_{\alpha}\wedge\omega_0^q = \alpha.
\]
\end{defn}

\begin{lemma}\label{IBP21}
The following identity holds.
\begin{eqnarray}\label{IBP31}
i\ddbar T_{\alpha} &=& i\ddbar\phi\wedge T_{\alpha} - 2\Re\langle i\dbar\dbar^*_{\phi}\alpha, \alpha\rangle
\nonumber\\
&+& | \dbar\gamma_{\alpha}|_{\om_0}^2 - | \dbar\alpha |_{\om_0}^2
+ | \dbar^*_{\phi}\alpha|_{\om_0}^2.
\end{eqnarray}
\end{lemma}

Now if we take $\phi = \Phi$ and $\omega =\omega_g$, then $i\ddbar\phi = \omega_g$ by K\"ahler-Einstein condition on $M$. The observation again is that integration by parts works on this identity, for all objects vanishing in a neighborhood of the divisor(compare to Lemma \ref{Bochner}). Therefore, we have the following integral equation.
\begin{proposition}\label{prop}
Suppose $\alpha$ is any $(n,1)$ form with value in $-K_X$, such that $\alpha$ vanishes in an open neighborhood $D_{\ep}$ of the divisor. Then
\begin{equation}\label{Boc1}
\int_X |\alpha|_{\om_g}^2 + \int_X |\dbar\gamma_{\alpha}|_{\om_g}^2= \int_X |\dbar\alpha |_{\om_g}^2 + |\dbar^*_{\Phi}\alpha|_{\om_g}^2.
\end{equation}
\end{proposition}

The hope is to apply this Bochner formula to the form $\alpha = \omega_g\wedge v_{\ep}$. Then we can estimate the $L^2$ norm of 
$\dbar v_{\ep}$, but there are some error terms on the RHS of equation (\ref{Boc1}). Fortunately, they are negligible in the following sense.

\begin{lemma}\label{lem1}
$v_{\ep} \rightarrow v_2$ in $L^2(\omega_g, \Phi)$ norm, and $\partial^{\Phi}v_{\ep}\rightarrow \partial^{\Phi}v_2$ in both $L^2(\om_0)$ and $L^2(\Phi)$ norm. In particular, $|| \partial\rho_{\ep}\wedge v_2 ||_{L^2(\Phi)}\rightarrow 0$. 
\end{lemma} 
\begin{proof}
It's easy to see $||v_{\ep} - v_2||_{L^2(\omega_g,\Phi)}$ converges to zero when $\ep$ decreases to zero, since it's controlled by $||\chi_{D_{\ep}}v_2||_{L^2(\omega_g,\Phi)}$, and the measure of its support $D_{\ep}$ converges to zero. The latter is also true, since
\begin{eqnarray}\label{conv}
\partial^{\Phi}(v_2-v_{\ep}) &=& \partial(\rho_{\ep} v_2) - \partial\Phi\wedge (\rho_{\ep}v_2)
\nonumber\\
&=& \partial\rho_{\ep}\wedge v_2 + \rho_{\ep}( \partial v_2  - \partial\Phi\wedge v_2)
\nonumber\\
&=& \partial\rho_{\ep}\wedge v_2 + \rho_{\ep} u_2.
\end{eqnarray}
Hence $$||\partial^{\Phi}(v_2-v_{\ep})||_{L^2(\om_0)}\rightarrow 0,$$ by Lemma \ref{Bo}. Now we may take a closer look at the term $\partial\rho_{\ep}\wedge v_2$.
By Lemma \ref{Bo} again, we can write 
\begin{equation}\label{new}
\partial\rho_{\ep}\wedge v_2 =  \frac{\ep\eta'}{\log\tau} \left( \frac{\partial s}{s}\wedge v_2 - \partial\psi\wedge v_2    \right).
\end{equation}
Put $v_2 = X^1d\hat{z}^1+ \sum_{k>1}X^k d\hat{z}^k $ locally, we have for $j>1$,
\begin{eqnarray}\label{lem101}
X^j &=& g^{j\ol 1}\dbar_{1}u_2 + \sum_{k>1}g^{j\ol k}\dbar_{k}u_2
\nonumber\\
&\sim & r^{1-\beta}\cdot r^{\beta-1} + O(1)
\nonumber\\
&\sim& O(1).
\end{eqnarray}
Then the only singular term is 
\begin{equation}\label{1100}
\frac{\partial s\wedge v_2}{s\log\tau} =  \frac{X^1}{r\log r},
\end{equation} 
where $r=|z^1|$. But 
\begin{eqnarray}\label{1000}
X^1 &=& g^{1\ol1}\dbar_{1}u_2 + \sum_{k>1}g^{1\ol k}\dbar_{k}u_2 
\nonumber\\
&\sim& r^{2-2\beta}\cdot r^{\beta -1} + r^{1-\beta}\cdot O(1)
\nonumber\\
&\sim& r^{1-\beta }.
\end{eqnarray}
where we used the condition $u_2\in  C^{2,\alpha}_{\beta}$. Finally, 
\begin{equation}
|\partial\rho_{\ep}\wedge v_2 |^2 e^{-\Phi} \sim \frac{\ep^2}{r^2(\log r)^2},
\end{equation}
whose $L^1$ norm converges to zero when $\ep\rightarrow 0$.
\end{proof}

\begin{lemma}
\label{em2}
$\dbar\rho_{\ep} \wedge \dbar u_2 \rightarrow 0$ in $L^2_{(n,2)}(\omega_g,\Phi)$.
\end{lemma}
\begin{proof}
Since $\omega_g$ is isometric to the model cone metric $\omega_{D}$, it's enough to prove locally near the divisor
\begin{equation}
\label{em201}
||\dbar\rho_{\ep}\wedge\dbar u_2||_{L^2(\omega_{D}, \Phi)} \rightarrow 0.
\end{equation}
We can compute it as  
$$\dbar\rho_{\ep}\wedge\dbar u_2 = \sum_{j<k} ( \rho_{,\ol k} u_{,\ol j} - u_{,\ol k}\rho_{,\ol j}) d\ol z^{j}\wedge d\ol z^{k}.$$
Put $A_{\ol j\ol k} = \rho_{,\ol k} u_{,\ol j} - u_{,\ol k}\rho_{,\ol j}$ locally, and then we have 
\begin{eqnarray}\label{em202}
|\dbar\rho_{\ep}\wedge\dbar u_2|^2_{\omega_{\beta}}e^{-\Phi} &=& e^{-\Phi}\sum_{j<k, m<l} g^{\ol j m}_{\beta} g_{\beta}^{\ol k l} A_{\ol j\ol k}\ol{A_{\ol {m}\ol {l}}}
\nonumber\\
& =& e^{-\Phi} ( \sum_{k>1}r^{2-2\beta}|A_{\ol1 \ol k}|^2 + \sum_{j>1}r^{2-2\beta}|A_{\ol j\ol1}|^2+ \sum_{1<j<k} |A_{\ol j\ol k}|^2 )
\nonumber\\
&\sim& \sum_{k>1}|A_{\ol1 \ol k}|^2 + \sum_{j>1}|A_{\ol j\ol1}|^2+ r^{2\beta -2}\sum_{1<j<k} |A_{\ol j\ol k}|^2.
\end{eqnarray}

Now note that $\dbar\rho_{\ep} = \frac{\ep\eta'}{\log\tau} \left( \frac{\partial \ol s}{\ol s} - \dbar\psi  \right) $, which implies for any $k>1$
\begin{eqnarray}\label{em203}
A_{\ol 1 \ol k} & = &   \rho_{,\ol k} u_{,\ol 1} - u_{,\ol k}\rho_{,\ol 1}
\nonumber\\
&\sim& \ep ( r^{\beta -1} + r^{-1}(\log r)^{-1} ).
\end{eqnarray}
And $|A_{\ol j\ol k}|$ is bounded by $\ep$ for any $1<j<k$. Therefore, 
$$ |\dbar\rho_{\ep}\wedge\dbar u_2|^2_{\omega_{g}}e^{-\Phi} \sim \frac{\ep^2}{r^2 (\log r)^2},$$
whose $L^1$ norm converges to zero when $\ep$ does.

\end{proof}

Equipped with these estimates, it's ready to prove our main theorem.

\begin{proof}[Proof of Theorem \ref{main3}]
First recall that by definition $v_{\ep} = \chi_{\ep}v_2$, and note that 
\begin{equation}\label{main301}
\dbar v_{\ep}\wedge\omega_g = \dbar ( \chi_{\ep}v_2\wedge\omega_g ) = -\dbar\rho_{\ep}\wedge \dbar u_2,
\end{equation}
which supports on an annuals region near the divisor. Then the Bochner formula, Proposition \ref{prop} says
\begin{equation}\label{main302}
 \int_X |\dbar v_{\ep}|^2_{\omega_g}e^{-\Phi} =  \int_X |\partial^{\Phi}v_{\ep}|^2e^{-\Phi} 
 - \int_X \omega_g\wedge v_{\ep}\wedge\ol v_{\ep} e^{-\Phi} + \int_X |\dbar v_{\ep}\wedge\omega_g|^2_{\omega_g}e^{-\Phi},
\end{equation}
by taking $\alpha = \omega_{g}\wedge v_{\ep}$.
Notice that the first term $|| \partial^{\Phi}v_{\ep}||_{L^2(\Phi)}$ on the RHS of equation (\ref{main302}) converges to $||u_2||_{L^2(\Phi)}$ by Lemma \ref{lem1}, and the last term can be estimated since
\begin{equation}\label{main303}
 ||\dbar v_{\ep}\wedge\omega_g||_{L^2(\omega_g, \Phi)} = || \dbar\rho_{\ep}\wedge\dbar u_2 ||_{L^2(\omega_g, \Phi)} \rightarrow 0,
\end{equation}
by Lemma \ref{em2}.
Finally we take the limit on both sides
\begin{equation}
0\leq \lim_{\ep}  \int_X | \dbar v_{\ep}|^2_{\omega_g}e^{-\Phi} = \int_X |u_2|^2e^{-\Phi} - \int_X |\dbar u_2 |^2_{\omega_g} e^{-\Phi}\leq 0,
\end{equation}
since by Donaldson \cite{MR2975584}, every eigenvalue $\lambda \geq 1$ for functions in the space $ C^{2,\alpha}_{\beta}$.
Therefore, $$\lim_{\ep} \int_X |\dbar v_{\ep}|^2_{\omega_g} e^{-\Phi}= 0$$ and $v_2$ is a holomorphic vector field.
\end{proof}

\subsection{Some identities on K\"ahler-Einstein cone manifolds}

\begin{lemma}
\label{lem-identity}
For any real valued functions $\varphi, \psi, \zeta \in  C^{2,\alpha}_{\beta}\cap C^{\infty}(M)$, assume $\varphi, \psi\in H_{\theta}$, we have 
\begin{equation}
\label{ibp111}
\Delta_{\theta} \langle \partial\zeta, \partial \varphi \rangle_{\theta} = \langle \ddbar\zeta , \ddbar\varphi \rangle_{\theta}
+\langle \d(\Delta_{\theta}\zeta), \d \varphi \rangle_{\theta},
\end{equation}
on $M$.
In particular, we have 
$$ (\Delta_{\theta} + 1) \langle \d\psi, \d \varphi \rangle_{\theta} = \langle \ddbar \psi, \ddbar \varphi\rangle_{\theta} 
= (\Delta_{\theta} + 1) \langle \d\varphi, \d \psi \rangle_{\theta},  $$
on $M$. And the following integral is finite:
\begin{equation}
\label{ibp222}
 -\int_{M}\varphi \langle \ddbar \zeta, \ddbar \psi \rangle_{\theta}\theta^n = \int_{M} (\varphi\psi - \langle \d \varphi, \d \psi \rangle_{\theta}) \xi \theta^n,
\end{equation}
where $\xi: =(\Delta_{\theta} +1)\zeta$.
\end{lemma}
\begin{proof}
The first equality is a point-wise computation, so we don't repeat it here. For the integral equality $(\ref{ibp222})$, notice that the RHS is always finite, since 
$\varphi\psi$, $ \langle \d \varphi, \d \psi \rangle_{\theta}$ and $\xi$ are all bounded function on $X$ thanks to the $ C^{2,\alpha}_{\beta}$ condition.
Then according to Lemma (2.3) in Bando-Mabuchi \cite{MR946233}, it is enough to prove the following integral equations hold, and the integrals are finite: 
\begin{equation}
\label{ibp333}
-\int_{M} \xi \d (\varphi\dbar \psi)\wedge n\theta^{n-1} = \int_{M} \varphi\d\xi\wedge\dbar\psi\wedge n\theta^{n-1},
\end{equation}
and 
\begin{equation}
\label{ibp555}
\int_{M} \varphi \Delta_{\theta} \langle \d\zeta, \d\psi  \rangle_{\theta} \theta^n = -\int_{M} \varphi  \langle \d\zeta, \d\psi  \rangle_{\theta} \theta^n.
\end{equation}
Here we invoke our cut off function $\chi_{\ep}$ again, and notice that the LHS of equation (\ref{ibp333}) is finite. Then we have 
$$\lim_{\ep\rightarrow 0} \int_{M} (\chi_{\ep}\xi) \d (\varphi\dbar \psi)\wedge n\theta^{n-1} = \int_{M} \xi \d (\varphi\dbar \psi)\wedge n\theta^{n-1} .  $$
But here we can apply integration by parts before taking the limit as
\begin{eqnarray}
\label{ibp666}
 - \int_{M} (\chi_{\ep}\xi) \d (\varphi\dbar \psi)\wedge n\theta^{n-1} &=& \int_{M} \d(\chi_{\ep}\xi)\wedge  (\varphi\dbar \psi)\wedge n\theta^{n-1}
\nonumber\\
&=&  \int_{M} \chi_{\ep} \d \xi \wedge  (\varphi\dbar \psi)\wedge n\theta^{n-1}\\
&+& \int_{M} \varphi\xi \text{tr}_{\theta} (\d\chi_{\ep} \wedge  \dbar \psi ) \theta^{n}\nonumber.
\end{eqnarray}

This is because we can view the integral is taken on the open subset $D_{\ep/3}$, 
and then one term in the integrant $(\chi_{\ep}\xi)$ vanishes identically near $\d D_{\ep/3}$. 
Now we can estimate the second term on the last line of above equation as
$$\int_{M} \varphi\xi \text{tr}_{\theta} (\d\chi_{\ep} \wedge  \dbar \psi ) \theta^{n} = 
\int_{D_{\ep} - D_{\ep/2}} \varphi\xi \text{tr}_{\theta} (\d\chi_{\ep} \wedge  \dbar \psi ) \theta^{n} \rightarrow 0, $$
since $$\text{tr}_{\theta} (\d\chi_{\ep} \wedge \dbar\psi) \theta^n \sim \ep (r^{2-\beta} \log r)^{-1}\text{ is }L^1$$ near the divisor (here we can use the local model metric $\omega_{D}$ instead of $\theta$ to compare thanks to the isometric property). Then we proved equation (\ref{ibp333}) by passing to limit. 

For equation (\ref{ibp555}), the RHS is obviously finite, and we use the cut off function to approximate as 
\begin{eqnarray}
\label{ibp2}
\int_{M} \chi_{\ep} \varphi \Delta_{\theta}\langle \d\zeta, \d \psi \rangle_{\theta} \theta^n  
&=& \int_{M} \Delta_{\theta}(\chi_{\ep}\varphi) \langle \d\zeta, \d \psi \rangle_{\theta} \theta^n
\nonumber\\
&=& \int_{M} \text{tr}_{\theta}( \varphi\ddbar\chi_{\ep} + \chi_{\ep}\ddbar\varphi      ) \langle \d\zeta, \d \psi \rangle_{\theta} \theta^n
\nonumber\\
&+& \int_{M}  \text{tr}_{\theta}( \d\varphi\wedge\dbar\chi_{\ep} + \d\chi_{\ep}\wedge \dbar \varphi ) \langle \d\zeta, \d \psi \rangle_{\theta} \theta^n.
\end{eqnarray}
Then we can estimate as before:
$$\int_{M} \text{tr}_{\theta}( \varphi\ddbar\chi_{\ep} )\langle \d\zeta, \d \psi \rangle_{\theta} \theta^n = 
\int_{D_{\ep} - D_{\ep/2}}\text{tr}_{\theta}( \varphi\ddbar\chi_{\ep} )\langle \d\zeta, \d \psi \rangle_{\theta} \theta^n  \rightarrow 0, $$
since $\langle \d\zeta, \d \psi \rangle_{\theta}$ is bounded, and $$\text{tr}_{\theta}( \varphi\ddbar\chi_{\ep} )\theta^n \sim \ep(r^2\log r^2)^{-1}\text{ is }L^1$$ near the divisor.
And then 
$$ \int_{M}  \text{tr}_{\theta}( \d\varphi\wedge\dbar\chi_{\ep} ) \langle \d\zeta, \d \psi \rangle_{\theta} \theta^n \rightarrow 0,$$
by the same reason. Therefore, the integral equality (\ref{ibp555}) follows by passing to the limit.

\end{proof}

\section{The continuity path}

Let $\om$ be a $C^\a_\b$ K\"ahler cone metric and let $\om_\vphi$ be a K\"ahler-Einstein cone metrics which satisfy \eqref{keg}. Additionally, in order to normalise the K\"ahler cone potential $\vphi$, we require it lies in $\cH_\b^0$.

We connect $\om_\vphi$ with $\om$ by the continuity pate $\vphi(t)$ satisfying the equation of currents
\begin{align}\label{Ricpath}
Ric(\om_{\vphi(t)})=t\om_{\vphi(t)}+(1-t)\om+2\pi(1-\b)[D].
\end{align}
It is obvious that $\vphi(t)=0$ is a trivial solution for any $0\leq t\leq1$ and it is the unique solution for any $0\leq t<1$, according to Proposition \ref{weak unique}. 

The trouble is at $t=1$, where the linearised operator is 
$$L_{\vphi(1)}u=\tri_\vphi u+u$$
which is no longer invertible and whose coefficient is the K\"ahler cone metric $\om_\vphi$. The kernel of $L_{\vphi(1)}$ is one to one corresponding to the holomorphic vector field, according to Section \ref{sec01}.
This difficulty is overcomed in Subsection \ref{Bifurcation} by extending Bando-Mabuchi's method to find a holomorphic transformation $\rho$ such that $$\rho^\ast\om_\vphi=\theta=\om_{\la_\theta}$$ and the linearised operator is invertible at such new K\"ahler-Einstein cone metric $\theta$. 

Recall that the formula of $\mathfrak f$ (see \eqref{spgeneral}) is 
\begin{align}\label{spgenerallater}
\mathfrak f=-\log(\frac{\om^n}{\om^n_0}|s|_h^{2(1-\b)})
-\delta|s|_h^{2\b}+h_0.
\end{align}
Written in the potential level, the continuity path becomes
\begin{align}\label{path}
\frac{\om_\vphi^n}{\om^n}=e^{\mathfrak f-t\vphi}
\end{align}
under the normalization condition for $0<t\leq 1$,
\begin{align}\label{norpath}
\int_Me^{\mathfrak f-t\vphi}\om^n=\int_M\om^n=V.
\end{align}
Using the smooth K\"ahler metric $\om_0$ as background metric, we also have
\begin{align}\label{pathsmooth}
\frac{\om_\vphi^n}{\om_0^n}=|s|_h^{2(\b-1)}e^{-t\vphi-\delta|s|_h^{2\b}+h_0}.
\end{align}

\subsection{Eigenvalues and openness on $[0,1)$}

\begin{prop}\label{weak unique}
Suppose that \eqref{path} has a solution $\vphi(s)$ at $t=s$ for some $0\leq s<1$. Then there exists a small constant $\epsilon>0$ such that \eqref{path} has a unique solution on $s\leq t\leq s+\epsilon$.
\end{prop}
\begin{proof}
We denote
\begin{align}\label{equ F}
F(\vphi,t)=\log\frac{\om_\vphi^n}{\om^n}-\mathfrak{f}+t\vphi.
\end{align}
It is a nonlinear operator from $C^{2,\a}_\b$ to $C^\a_\b$.
So the linearisd operator is 
\begin{align*}
L_{\vphi(t)}u=\tri_\vphi u+tu
\end{align*}
also from $C^{2,\a}_\b$ to $C^\a_\b$. The condition \eqref{norpath} gives the condition for $0<t<1$,
\begin{align}\label{norpathlin}
\int_M(\vphi+t u)\om_\vphi^n=0.
\end{align}

In order to solve the linear equation defined by the linearised operator, we require the weak solution theory and Donaldson's regularity estimate of the linear equation with K\"ahker cone metric as the coefficients of the leading term. The details and more information could be found in Calamai-Zheng \cite{Calamai-Zheng}. While, the following Proposition \ref{eigenvalue} tells that the linearisation equation has no kernel.
 \end{proof}

Along the continuity path, the volume form $\omega_{\vphi(t)}^n$ can be viewed as a metric on $-K_X$ as
$$\psi(t):=-\log \omega_{\vphi(t)}^n.$$
Then the Laplacian operator for the metric $\omega_{\vphi(t)}$ can be written as 
$$ \Delta_{\omega_{\vphi(t)}} = \dbar^*_{\psi(t)}\dbar,$$
and it can be viewed as an operator acting on $(n,0)$ forms with value in $-K_X$. 
According to Lemma \ref{IBP21}, the Bochner formula reads as 
\begin{equation}
\label{Bochner}
\int_X |\dbar\gamma_{\alpha}|_{\omega_{\vphi(t)}}^2= \int_X |\dbar\alpha |_{\omega_{\vphi(t)}}^2 + \int_X |\dbar^*_{\psi(t)}\alpha|_{\omega_{\vphi(t)}}^2 - \int_X Ric(\alpha,\alpha),
\end{equation}
where $\alpha$ is any $(n,1)$ form with value in $-K_X$ vanishing on $D_{\ep}$. 

Now we assume that $u(t)$ is an eigenfunction of $\Delta_{\omega_{\vphi(t)}}$ with eigenvalue $\lambda$, and we also assume 
$u(t)$ is real-valued and belongs to the H\"older space $ C^{2,\alpha}_{\beta}$. That is, there exists an $(n-1,0)$ form $v(t)$ with value in $-K_X$ satisfying 
\begin{equation}
\label{1181}
\left\{ \begin{array}{rcl}
\omega_{\vphi(t)}\wedge v(t) &=& \dbar u(t),\\
\partial^{\psi(t)}v(t) &=& \lambda u(t).
\end{array}\right.
\end{equation}

Then we are going to prove the following statement.
\begin{prop}\label{eigenvalue}
For all $\lambda \leq t$, there is no such eigenfunction $u(t)$ for $\lambda$.
\end{prop}
\begin{proof}
We assume that $u(t)$ exists for some $\lambda \leq t$, and $v(t)$ is defined as equation (\ref{1181}).
Define $$v_{t,\ep} := \chi_{\ep}\cdot v(t) $$ and choose $\alpha = \dbar u(t)$ in the Bochner formula (\ref{Bochner}), we have the following identity
\begin{eqnarray}
\label{3020}
 \int_X | \dbar v_{t, \ep}|_{\omega_{\vphi(t)}}^2 e^{-\psi(t)} &
 =&  \int_X | \partial^{\psi(t)}v_{t, \ep}|_{\omega_{\vphi(t)}}^2e^{-\psi(t)} 
 - \int_X Ric(\omega_{\vphi(t)}) \wedge v_{t,\ep}\wedge\ol{v}_{t,\ep} e^{-\psi(t)} 
 \nonumber\\
 & +& \int_X |\dbar v_{t, \ep}\wedge\omega_{\vphi(t)}|_{\omega_{\vphi(t)}}^2 e^{-\psi(t)}
 \nonumber\\
 & < & \int_X | \partial^{\psi(t)}v_{t, \ep}|_{\omega_{\vphi(t)}}^2e^{-\psi(t)} - t \int_X \omega_{\vphi(t)} \wedge v_{t,\ep}\wedge\ol{v}_{t,\ep} e^{-\psi(t)} \\
 &+& \int_X |\dbar v_{t, \ep}\wedge\omega_{\vphi(t)}|_{\omega_{\vphi(t)}}^2 e^{-\psi(t)}.
 \nonumber
\end{eqnarray}

Thanks to Lemma \ref{lem1} and \ref{em2}, we can take the limit when $\ep\rightarrow 0$ as
\begin{eqnarray}
0 &<& \lambda^2\int_{M} |u(t)|_{\omega_{\vphi(t)}}^2 e^{-\psi(t)} - t \int_{M} |\dbar u(t)|_{\omega_{\vphi(t)}}^2 e^{-\psi(t)}  
\nonumber\\
&=& \lambda(\lambda - t) \int_X | u(t) |_{\omega_{\vphi(t)}}^2 e^{-\psi(t)} \leq 0.
\end{eqnarray}
In the last line we used the integration by parts, since $u(t), \varphi(t)\in  C^{2,\alpha}_{\beta}$.
Thus we get the contradiction and the poposition follows.
\end{proof}

\begin{lem}\label{IJdecreasing}
The $I-J$ is non-decreasing along the continuity path.
\end{lem}
\begin{proof}
Along the path \eqref{path}, we have $\tri_\vphi\dot\vphi+\vphi+t\dot\vphi=0$.
Since $\dot{\vphi}$ is $C^{2,\a}_\b$, we have the integration by parts, i.e.
\begin{align*}
\frac{d}{dt}(I-J)
&=-\frac{1}{V}\int_{M}\varphi\triangle_\varphi\dot\varphi\omega_\varphi^{n}  \\
&=\frac{1}{V}\int_{M}(\tri_\vphi\dot\vphi)^2-t|\p\dot\vphi|^2\omega_\varphi^{n}.
\end{align*}
Thus from Proposition \ref{eigenvalue}, we have
$$\frac{d}{dt}(I-J)\geq 0.$$
\end{proof}

\subsection{Approximation of the continuity path}
We now prove the approximation of the continuity path.
Recall that the potential equation along the continuity path satisfies,
\begin{align} \label{originalpath}
\frac{\om_\vphi^n}{\om_0^n}=|s|_h^{2(\b-1)} e^{-t\vphi-\delta|s|^{2\b}_h+h_0}
\end{align}
under the normalisation condition for $0<t\leq 1$,
\begin{align}\label{norpath1}
\int_M|s|_h^{2(\b-1)} e^{-t\vphi-\delta|s|^{2\b}_h+h_0}\om_0^n=V.
\end{align}

From now on, we fix $\tau$ to be a small strictly positive constant which is less than $1$.

\begin{thm}\label{approximation Theorem}
Along the continuity path $\{\om_{\vphi(t)};\tau\leq t\leq 1\}$, the following holds.
The path $\om_{\vphi(t)}$ is the Gromov-Hausdorff limit of a sequence of smooth K\"ahler metrics $\om^i_{\vphi(t)}$ with uniformly bounded diameter and non-negative Ricci curvature. 
\end{thm}
\begin{proof}
This is an adaption of Theorem 1.1 in \cite{MR3264766} to our continuity path $\om_\vphi$. 
We omit the index $t$ of $\vphi(t)$ in the following proof, since it is obvious.

\textbf{Step 1.} Since $\om_\vphi$ is in $L^p$, we choose a sequence of smooth volume form $\eta_\eps$, which $L^p$ converges to $\om_\vphi$. Then Yau's resolution of Calabi conjecture provides a K\"ahler potential $\vphi_\eps$, such that
\begin{align}
\om^n_{\vphi_\eps}=\eta_\eps.
\end{align}
From \cite{MR1618325}, $\vphi_\eps$ has uniform $C^\a$ bound and thus converges to $\vphi$ in $C^{\a'}$ for any $\a'<\a$ as $\eps\rightarrow 0$.

\textbf{Step 2.} Adjusting by a constant such that
\begin{align}\label{appronorpath}
\int_M(|s|^2_h+\eps)^{\b-1} e^{-t\vphi_\eps-\delta|s|^{2\b}_h+h_0}\om_0^n=V,
\end{align}
then replacing $\vphi$ on the right hand side of \eqref{originalpath} with $\vphi_\eps$, we have
\begin{align} \label{appropath}
\frac{\om_{\psi_\eps}^n}{\om_0^n}=(|s|^2_h+\eps)^{\b-1} e^{-t\vphi_\eps-\delta|s|^{2\b}_h+h_0}.
\end{align}
Again, Yau's celebrated work gives a solution $\psi_\eps$, which satisfies this equation. 
Again, from \cite{MR1618325}, $\psi_\eps$ has uniform $C^\a$ bound and converges to $\psi_0$ in $C^{\a'}$ for any $\a'<\a$ as $\eps\rightarrow 0$.

\textbf{Step 3.} We compute the Ricci curvature of $\om_{\psi_\eps}$. With the formula
\begin{align}
i\p\bar\p\log(f+\eps)\geq
\frac{f}{f+\eps}i\p\bar\p\log f,
\end{align} we have
\begin{align*}
Ric(\om_{\psi_\eps})
&=Ric(\om_0)-i\p\bar\p h_0+(1-\b)i\p\bar\p\log(|s|^2_h+\eps)+ti\p\bar\p\vphi_\eps+\delta i\p\bar\p |s|^{2\b}_h\\
&=\om_0-(1-\b)i\p\bar\p\log h+(1-\b)i\p\bar\p\log(|s|^2_h+\eps)+ti\p\bar\p\vphi_\eps+\delta i\p\bar\p |s|^{2\b}_h\\
&=t\om_{\vphi_\eps}+(1-t)\om-(1-\b)i\p\bar\p\log h+(1-\b)i\p\bar\p\log(|s|^2_h+\eps)\\
&\geq t\om_{\vphi_\eps}+(1-t)\om-(1-\b)i\p\bar\p\log h+(1-\b)\frac{|s|^2_h}{|s|^2_h+\eps}i\p\bar\p\log |s|^2_h\\
&\geq t\om_{\vphi_\eps}+(1-t)\om
+(1-\b)[-i\p\bar\p\log h+\frac{|s|^2_h}{|s|^2_h+\eps}i\p\bar\p\log |s|^2_h].
\end{align*}

By our choice of $h$ which is a Hermitian metric on $[D]$, we have \begin{align}-i\p\bar\p\log h+\frac{|s|^2_h}{|s|^2_h+\eps}i\p\bar\p\log |s|^2_h\geq 0.\end{align} Thus we have for any $\eps\in(0,1]$, $$Ric(\om_{\psi_\eps})\geq t\om_{\vphi_\eps}\geq0.$$

To sum up, we have proved
\begin{prop}
The approximation sequence $\om_{\psi_\eps}$ has non-negative Ricci curvature.
\end{prop}

\textbf{Step 4.} 
We prove the rough second order estimate of $\psi_\eps$. 
From the Chern-Lu inequality (see Section \ref{Higher order estimate})
\begin{align*}
\tri_{\psi_\eps}(\log \tr_{\om_{\psi_\eps}}\om_0-C{\psi_\eps})\geq
\frac{R_{\psi_\eps}^{i\bar j}g_{0i\bar j}-g_{\psi_\eps}^{i\bar j}g_{\psi_\eps}^{k\bar l}Rm(\om_0)_{i\bar jk\bar l}}{\tr_{\om_{\psi_\eps}}\om_0}-Cn+C\tr_{\om_{\psi_\eps}}\om_0.
\end{align*}
Since $Ric_{\psi_\eps}\geq0$ when $0\leq t\leq \b$, we have
\begin{align*}
\tri_{\psi_\eps}(\log \tr_{\om_{\psi_\eps}}\om_0-C{\psi_\eps})
\geq
[-\max_X Rm(\om_0)+C] \cdot \tr_{\om_{\psi_\eps}}\om_0-Cn.
\end{align*}
Choosing $C=\max_X Rm(\om_0)+1$,
we have the lower bound of $\om_{\psi_\eps}$
\begin{align*}
tr_{\om_{\psi_\eps}}\om_0\leq C(\osc{\psi_\eps}).
\end{align*}
While, we also have the upper bound
\begin{align*}
\tr_{\om_0}\om_{\psi_\eps}\leq[\frac{\om^n_{\psi_\eps}}{\om_0^n} \cdot tr_{\om_{\psi_\eps}}\om_0]^n
=[(|s|^2_h+\eps)^{\b-1} e^{-t\vphi_\eps+h_0} tr_{\om_{\psi_\eps}}\om]^n.
\end{align*}
Thus there is a constant $C$ (independent of $t$) such that for any $\eps\in(0,1]$, 
\begin{align}
C^{-1}\om_0\leq \om_{\psi_\eps}\leq \frac{C}{(|s|^2_h+\eps)^{1-\b}} \om_0.
\end{align}
Then we have the uniform diameter bound of $\om_{\psi_\eps}$ by measuring the length in a small neighbourhood of $D$ under $\om_0$ and outside under $\om_{\psi_\eps}$. The length outside is bounded by using the inequality above 
in conclusion, we arrive at the following proposition.
\begin{prop}
For any $0\leq t\leq 1$, the approximation sequences $\om_{\psi_\eps}$ have uniformly bounded diameter.
\end{prop}
\textbf{Step 5.} 
We show that the limits from \textbf{Step 2.} have the relation,
\begin{align*}
\psi_0=\vphi+constant.
\end{align*}
In order to prove this identity, we apply the formula
\begin{align*}
\int_X (\vphi_\eps-\psi_\eps)(\omega_{\psi_\eps}^n-\omega_{\vphi_\eps}^n)
=\int_X i\p(\vphi_\eps-\psi_\eps)\wedge\bar\p(\vphi_\eps-\psi_\eps)\wedge \sum_{k=0}^{n-1}\om^k_{\psi_\eps}\wedge\om_{\vphi_\eps}^{n-1-k}.
\end{align*}
Cutting out a small neighbourhood $D_\delta$ for arbitrary $\delta>0$, we have the RHS 
\begin{align*}
\geq \int_{X\setminus D_\delta} i\p(\vphi_\eps-\psi_\eps)\wedge\bar\p(\vphi_\eps-\psi_\eps)\wedge \om_{\psi_\eps}^{n-1}.
\end{align*}
From \textbf{Step 4}, we further have
\begin{align*}
\geq C \int_{X\setminus D_\delta} i\p(\vphi_\eps-\psi_\eps)\wedge\bar\p(\vphi_\eps-\psi_\eps)\wedge \om_{0}^{n-1}.
\end{align*}
While, LHS converges to $0$ as $\eps\rightarrow 0$. Thus $\delta$ is arbitrary, we have $\psi_0=\vphi$ up to a constant on $M$.

\textbf{Step 6.} 
Proposition 2.5 in \cite{MR3264766} tells us that $\om_{\psi_\eps}$ Gromov-Hausdorff converges to $\om_\vphi$ as $\eps\rightarrow0$. Cheeger-Colding \cite{MR1815410} implies there is a minimising geodesic in $M$ such that its length is close to the diameter of $X$.
Then when $t\geq \tau$, $\om_\vphi$ has diameter bound $\pi\sqrt\frac{m-1}{\tau}$ due to Myers' theorem. Thus we could choose small $\eps$ such that the sequence has diameters bounded by $2\pi\sqrt\frac{m-1}{\tau}$.
\end{proof}

The next Sobolev inequality along the continuity path will be used in this paper.

\begin{thm}\label{Sobolev inequalitybig}
Let $\om_\vphi$ lies in the continuity path $\{\om_{\vphi(t)};\tau< t \leq 1\}$.
For any $1\leq q<m$, there exists a uniform
constant $A=A(n,q,V,\tau)$ such that for any $w\in W^{1,q}(M)$,
\begin{align*}
\|w\|_{p;\om_{\vphi(t)}} \leq
A\|\nabla
w\|_{q;\om_{\vphi(t)}}+Vol(M,\om_{\vphi(t)})^{\frac{-1}{m}}\|w\|_{q;\om_{\vphi(t)}},
\end{align*} where the constant $p$ is defined by  $\frac{1}{p}+\frac{1}{m}=\frac{1}{q}.$\end{thm}
\begin{proof}
We cite the Sobolev inequality by Croke \cite{MR608287}, Gallot \cite{MR999971,MR976219} and Ilias \cite{MR699492}. Let  $(M,g)$ be a $m$-dimensional compact
Riemannian manifold with Ricci curvature, volume and diameter satisfying
\begin{align}
Ric\geq (m-1)kg, \quad Vol(g)\geq V \text{ and } diam(M,g)\leq d.
\end{align}In which, $k$, $V>0$, $d>0$ are real numbers.
For any $1\leq q<m$, there exists
constant $A=A(n,q,k,V,d)$ such that for any $w\in W^{1,q}(M,g)$,
\begin{align*}
\|w\|_{p;g} \leq
A\|\nabla
w\|_{q;g}+Vol(M,g)^{\frac{-1}{m}}\|w\|_{q;g},
\end{align*} where the constant $p$ is defined by  $\frac{1}{p}+\frac{1}{m}=\frac{1}{q}.$

When $\tau< t\leq1$, given $w\in  W^{1,q}(M,\om_\vphi)$, we see that $w$ also stays in $W^{1,q}(M,\om^i_\vphi)$. Then we apply this inequality to the approximation sequence $\om^i_\vphi$ which have uniformly non-negative Ricci curvature and uniform diameter bounded by $d=2\pi\sqrt{\frac{m-1}{\tau}}$. Actually, on the regular part $M$, our sequence smoothly converges to $\vphi$ from the construction, so the conclusion follows from Lebesgue's dominated convergence theorem and the fact that $D$ is a measure zero set.
\end{proof}


\subsection{Apriori estimates}\label{Apriori estimates}

\subsubsection{Zero order estimate}\label{Zero order estimate title}

We prove the zero order estimate
by the adaption of the De Giorgi iteration, which is an improvement of Proposition 2.8 in \cite{MR3229802} by H. Li and the second author. 

\begin{prop}\label{inte est}
Assume that we have the following Sobolev inequality with respect to a K\"ahler cone metric $\om$, for any $w \in W^{1,2}(\om)$,
$$\|w\|^2_{2^\ast}\leq C_S(\om)(\|\Na w\|^2_2+\|w\|^2_2).$$
We say $v$ is a $W^{1,2} $ sub-solution of the linear equation in the weak sense, i.e. for any $\eta\in C^{2,\a}_\b$,
\begin{align}\label{linear equ ge sub}
\int_M(\p v, \p\eta)_{\om} \om^n\leq -\int_Mf\eta\om^n.
\end{align} Moreover, we assume that $f\in L^{\frac{p}{2}}$
with $p> 2n$ and let $$\tilde v=v-\frac{1}{V}\int_M v\,\om^n,$$ then there exits a constant $C$ depending on the Sobolev constant $C_S(\om)$ with respect to $\om$ such that
\begin{align}\label{glob bound claim}
\sup_M \tilde v \leq C(\| f \|_{p^\ast} +\| \tilde v\|_{1}).
 \end{align}
In which, $p^\ast=\frac{2np}{2n+p}$ and all the $L^p$-norms, including in the following proof, are regarding to the mearsure $\frac{\om^n}{V}$.
\end{prop}
\begin{proof}
We denote by $u=(\tilde
v-k)_+$  the positive part of  $\tilde v-k$ for any constant $k$ and 
set $$A(k)=\{x\in M\vert \tilde v(x)>k\}$$ where $u$ is positive. 

We first substitute $\eta$ in \eqref{linear
equ ge sub} with $u$ on both sides
\begin{align}\label{linear equ ge1}
\int_M |\Na u|^2\om^n\leq -\int_M u f \om^n.
\end{align}
Then applying the H\"older's inequality to the right hand side, we get its upper bound
\begin{align}\label{linearRHS}
\|u\|_{2^\ast}\cdot \|f\|_{p^\ast}\cdot |A(k)|^r.
\end{align}
In which, $$m=2n,\quad 2^\ast=\frac{2m}{m-2},\quad p^\ast=\frac{mp}{m+p},\quad r=\frac{1}{2}-\frac{1}{p}.$$

We next use the Sobolev inequality with respect to $\om$,
$$\|u\|^2_{2^\ast}\leq C_S(\om)(\|\Na u\|^2_2+\|u\|^2_2).$$
Here the norms are measured with respect to the metric $\om$. The second term in the right hand side of the Sobolev inequality is bounded by using the H\"older inequality $$\|u\|_2\leq \|u\|_{2^\ast}\cdot
|A(k)|^\frac{1}{m}.$$ 
While, the first term in the right hand side of the Sobolev inequality is estimated by using \eqref{linearRHS}, which is derived from the equation. Thus we obtain that  
\begin{align} \label{eqa2}
\|u\|^2_{2^\ast}\leq
C_S(\om)(\|u\|_{2^\ast}\cdot\|f\|_{p^\ast}\cdot |A(k)|^r+\|u\|^2_{2^\ast}\cdot
|A(k)|^\frac{2}{m}).
\end{align}

We then show that how to choose a $k_0$ 
such that for any $k\geq k_0$, 
\begin{align}\label{areabound}
|A(k)|^\frac{2}{m}\leq
\frac{1}{2C_S(\om)}. 
\end{align}
In order to choose $k_0$, we separate two cases.
On case is 
\begin{align}
\|\tilde v\|_2^2\leq \|\Na \tilde v\|_2^2 .\label{eqa1} 
\end{align} The
right hand side of (\ref{eqa1}) is bounded from the
inequality \eqref{linear equ ge sub} with $\eta=\tilde v$ via applying the
H\"older inequality to its right hand side, $$\|\Na\tilde v\|_2^2\leq \|\tilde
v\|_2\|f\|_2.$$ Thus we have
 $$\|\tilde v\|_2 \leq \|f\|_2\leq V^{\frac{2mp}{mp-2m-2p}}\cdot \|f\|_{p^\ast}.$$ 
 
The other case is \begin{align*}
\|\tilde v\|_2^2\geq \|\Na \tilde v\|_2^2 .
\end{align*}
The Sobolev inequality immediately implies that
\begin{align*}\|\tilde v\|^2_{2^\ast}\leq 2C_S(\om)\|\tilde v\|^2_2.
\end{align*}
We then apply the interpolation inequality to the right hand side of the inequality above,
\begin{align*}
\|\tilde v\|^2_{2^\ast}\leq 2C_S(\om)\|\tilde v\|^{\frac{2}{n+1}}_1 \|\tilde v\|_{2^\ast}^{\frac{2n}{n+1}}.
\end{align*}
Thus 
\begin{align*}
\|\tilde v\|_{2}\leq V^{\frac{1}{n}}\|\tilde v\|_{2^\ast}\leq 2C_S(\om)\|\tilde v\|_1 .
\end{align*}

Combining two cases, we see that 
\begin{align*}
\|\tilde v\|_{2}\leq C_5(\|f\|_{p^\ast}+\|\tilde v\|_1).
\end{align*}
In which, $C_5=C_S(\om)+V^{\frac{2mp}{mp-2m-2p}}$.
From the definition, $$k_0^2 |A(k_0)|\leq
\|\tilde v\|_2^2.$$ 
Thus we choose $$k_0^2=C_5^2(\|f\|_{p^\ast}+\|\tilde v\|_1)^2
(2C_S(\om))^{\frac{m}{2}}$$ so that $$|A(k_0)|^\frac{2}{m}\leq\frac{1}{2C_S(\om)}$$ and \eqref{areabound} is proved.

For any $k\geq k_0$, we absorb the second term of right hand side of \eqref{eqa2} by the left hand side, when applying \eqref{areabound}, thus we obtain,
$$\|u\|_{2^\ast}\leq 2C_S(\om)\cdot\|f\|_{p^\ast}\cdot |A(k)|^{r}.$$
The inverse inequality follows from the definition of $A(h)$ and $u$, when $h>k\geq k_0$,
$$\|u\|_{2^\ast}\geq (h-k)\cdot |A(h)|^{\frac{1}{2^\ast}}.$$
At last, combining these two inequalities to obtain the iteration inequality
$$(h-k)\cdot |A(h)|^{\frac{1}{2^\ast}}\leq  2C_S(\om) \cdot \|f\|_{p^\ast}\cdot |A(k)|^{r}.$$
and then applying the iteration lemma (see \cite{MR1814364}),
we have $$A(k_0+d)=0$$ for a constant $d= 2C_S(\om) \|f\|_{p^\ast}$. In other words,
\begin{align*}\tilde v&\leq k_0 + d\\
&\leq C_5(\|f\|_{p^\ast}+\|\tilde v\|_1)
(2C_S(\om))^{\frac{m}{4}}+ 2C_S(\om) \|f\|_{p^\ast}.
\end{align*} Therefore, the proposition is proved.
\end{proof}

We then apply the Proposition above to the following equation
\begin{align}\label{upperequation}
n+\tri_\om\vphi\geq0.
\end{align}

\begin{cor}(Upper estimate)\label{zeroup}
There exists a constant $C$ depending on $V$, $n$, the Sobolev and Poincar\'e constant of the background metric $\om$ such that
\begin{align*}
\sup_M\varphi-\frac{1}{V}\int_M\varphi\omega^n\leq
C.
\end{align*}
\end{cor}
\begin{proof}
Let $\tilde\vphi=\vphi-\frac{1}{V}\int_M\vphi\om^n$. The \eqref{upperequation} is well-defined on the regular part $M$, we need to transform it into the integration form.
Since $\vphi\in C^{2,\a}_\b$, the integration by parts, Lemma 2.1 in Calamai-Zheng \cite{Calamai-Zheng} provides that \eqref{upperequation} could be transformed to, for any $\eta\in C^{2,\a}_\b$,
\begin{align}\label{intelowerequation}
\int_M (\p\tilde\vphi,\p\eta)_{\om} \om^n\leq n \int_M \eta\om^n.
\end{align}
Thus Proposition \ref{inte est} implies that there is a constant $C$ depending on the Sobolev constant of $\om$ such that 
\begin{align*}
\sup_M \vphi-\frac{1}{V}\int_M\vphi\,\om^n\leq C(n +\| \tilde \vphi\|_{1:\om}).
 \end{align*}
We would prove that $\| \tilde \vphi\|_{1;\om}$ is bounded.
Replacing $\eta$ with $\tilde\vphi$ in \eqref{intelowerequation}, we have
\begin{align*}
||\p\tilde\vphi||_{2;\om}^2\leq n ||\tilde\vphi||_{1;\om}.
\end{align*}
The Poincar\'e inequality implies that there is Poincar\'e constant $C_P$ such that
\begin{align*}
||\tilde\vphi||_{2;\om}^2\leq C_P ||\p\tilde\vphi||_{2;\om}^2.
\end{align*}
While, the H\"older inequality gives that
\begin{align*}
||\tilde\vphi||_{1;\om}\leq V^{\frac{1}{2}}||\tilde\vphi||_{2;\om}.
\end{align*}
Combining all these three inequalities together, we have $||\tilde\vphi||_{2;\om}$ is bounded, then from the the last inequality, so is $||\tilde\vphi||_{1;\om}$. Therefore, we have proved \eqref{zeroup}.
\end{proof}

We then apply Proposition \ref{inte est} to the following equation
\begin{align}\label{lowerequation}
n-\tri_\vphi\vphi>0.
\end{align}
\begin{cor}(Rough lower estimate)\label{zerolowbig}
Assume that $\vphi\in C^{2,\a}_\b$ and $\om_\vphi$ lies in the continuity path $\{\om_{\vphi(t)};\tau< t \leq 1\}$.
There exists a constant $C$ depending on $\sup_{\tau<t\leq 1} C_S(\om_{\vphi(t)}), V,n$ such that 
\begin{align*}
\inf_M\vphi-\int_M\vphi\,\om_\vphi^n\geq - C(1+\| \tilde\vphi\|_{1;\om_\vphi}),\quad t\in[\tau,1].
\end{align*}
\end{cor}
\begin{proof}
Since when $\vphi\in C^{2,\a}_\b$, we apply integration by parts in Calamai-Zheng \cite{Calamai-Zheng}, to \eqref{lowerequation}, then obtain for any $\eta\in C^{2,\a}_\b$,
\begin{align*}
\int_M (\p\vphi,\p\eta)_{\om_\vphi} \om_\vphi^n>-n \int_M \eta\om_\vphi^n.
\end{align*}
We use Proposition \ref{inte est}, replacing $\om$ with $\om_\vphi$ and obtain the lower bound of $\tilde\vphi=\vphi-\frac{1}{V}\int_M\vphi\om^n_\vphi$.

For any $\tau\leq t\leq 1$ along the continuity path, we have the Sobolev inequalities of $\om_{\vphi(t)}$ (Theorem \ref{Sobolev inequalitybig}) with uniform Sobolev constant, i.e. the Sobolev constants $C_S(\om_{\vphi(t)})$ have a uniform bound. Thus we have obtained the conclusion.
\end{proof}

\begin{prop}\label{Zero order estimate}(Zero order estimate)
We are given a small fixed positive number $0<\tau<1$ to be determined in the proof.
There exists a constant $C$ depending on $\tau, V, n$, the Sobolev and Poincar\'e constant of the fixed background metric $\om$, and the uniform Sobolev constant of $\om_{\vphi(t)}$ for $\tau<t\leq 1$ i.e. $\sup_{\tau<t\leq 1} C_S(\om_{\vphi(t)})$ such that
\begin{align}\label{zeroroughbig}
\osc(\varphi)\leq C\cdot(I(\om, \om_{\varphi})+1),\quad t\in[0,1].
\end{align}
\end{prop}
\begin{rem}
We do not need the Poincar\'e constant of $\om_\vphi$ here.
\end{rem}
\begin{proof} 
There are two cases.

\textbf{Case 1.}  Thanks to the linearised operator of the continuity path is invertible at $t=0$, (Proposition \ref{weak unique}), we could choose a sufficient small $0\leq t\leq \tau$ to have a uniform zero order estimate of the solution $\vphi$ on $[0,\tau]$.

\textbf{Case 2.}  When $\tau\leq t\leq 1$, putting the upper estimate (\corref{zeroup}) and the rough lower estimate (\corref{zerolowbig}) together, we arrive at
\begin{align*}
\osc(\varphi)\leq I(\om, \om_{\varphi})+C\cdot(1+\|\tilde \vphi\|_{1;\om_\vphi}).
\end{align*}
In order to bound $\| \tilde\vphi\|_{1;\om_\vphi}$, we calculate
\begin{align*}
\| \tilde\vphi\|_{1;\om_\vphi}=\int_M|\tilde\vphi|\omega_\vphi^n
&\leq\int_M|\sup_M\tilde\vphi-\tilde\vphi|\omega_\vphi^n+\int_M|\sup_M\tilde\vphi|\omega_\vphi^n\\
&= V\sup_M\tilde\vphi+\int_M\tilde\vphi\omega_\vphi^n+V\sup_M\tilde\vphi\\
&\leq2V\sup_M\tilde\vphi=2V(I(\om,\om_\vphi)+C).
\end{align*}
At the last step, we use \corref{zeroup},
$\sup_M\varphi\leq\frac{1}{V}\int_M\varphi\omega^n+
C$ again. In this case, for any $\tau\leq t\leq 1$ the Sobolev constants of $\om_\vphi$ (Theorem \ref{Sobolev inequalitybig}) are uniform and depend on $\tau, V, n$.

The bound of $I$ follows from the equivalence of the $I$ and $J$ functional and the monotonicity of $I-J$ along the continuity path from the \lemref{IJdecreasing}.
Therefore, the lemma is proved.
\end{proof}

\subsubsection{Higher order estimates}\label{Higher order estimate}
In order to derive the second order estimate, we follow the proof of Yau's Schwarz lemma by applying the Chern-Lu formula \cite{MR0486659}.
We derive the formula of $$A:=\tr_{\om_\vphi}\om - C\vphi= n- \Delta_\vphi \vphi - C\vphi .$$ 
We compute,
\begin{align}\label{lem_3_interior_laplacian_estimate}
 \Delta_\vphi (\tr_{\om_\vphi}\om)=
R_\vphi^{i\bar j} g_{i\bar j} - {g_\vphi}^{i\bar j}{g_\vphi}^{k\bar l}R_{i\bar j k \bar l}
-g^{i \bar j}{g_\vphi}^{k \bar l}{g_\vphi}^{p\bar q}
\p_{\bar l}{g_\vphi}_{p\bar j} \p_k {g_\vphi}_{i \bar q}\; .
\end{align}
Here $R_{i\bar j k \bar l}$ is the Riemannian curvature of the background metric $\om$ and $R_\vphi^{i\bar j}$ is the Ricci curvature of $\om_\vphi$.
The Schwarz inequality implies
\begin{align}\label{la sch}
g_\vphi^{k\bar l}\p_kg_\vphi^{i\bar j}g_{i\bar j}\p_{\bar l}g_\vphi^{p\bar q}
g_{p\bar q}
\leq
-(g_\vphi^{k\bar l}g_\vphi^{p\bar j}g_{i\bar j}\p_{\bar l}
g_{\vphi p\bar q}\p_{k}g_\vphi^{i\bar q})(g_\vphi^{i\bar j}g_{i\bar j})\;.
\end{align}
We apply \eqref{lem_3_interior_laplacian_estimate} and \eqref{la sch} to obtain
\begin{align*}
\tri_\vphi[\log tr_{\om_\vphi}\om]
&=\frac{\tri_\vphi(tr_{\om_\vphi}\om)}{tr_{\om_\vphi}\om}-
\frac{g_\vphi^{k\bar l}\p_kg_\vphi^{i\bar j}g_{i\bar j}\p_{\bar l}g_\vphi^{p\bar q}
g_{p\bar q}}{(tr_{\om_\vphi}\om)^2}\\
&\geq\frac{R_\vphi^{i\bar j}g_{i\bar j}-g_\vphi^{i\bar j}g_\vphi^{k\bar l}R_{i\bar jk\bar l}}{tr_{\om_\vphi}\om}\;.
\end{align*}
Thus
\begin{align*}
\tri_\vphi(\log \tr_{\om_\vphi}\om-C\vphi)\geq
\frac{R_\vphi^{i\bar j}g_{i\bar j}-g_\vphi^{i\bar j}g_\vphi^{k\bar l}R_{i\bar jk\bar l}}{\tr_{\om_\vphi}\om}-Cn+C\tr_{\om_\vphi}\om.
\end{align*}
Since along the path, we have
$Ric_\vphi>0$ and $Rm(\om)$ has upper bound.
Applying the cone maximum principle, we have the lower bound of $\om_\vphi$
\begin{align*}
tr_{\om_\vphi}\om\leq C(\osc\vphi
, \sup_M R_{i\bar j k\bar l}).
\end{align*}
While, we also have its upper bound
\begin{align*}
\tr_\om\om_\vphi\leq[\frac{\om^n_\vphi}{\om^n} tr_{\om_\vphi}\om]^n
=[e^{\mathfrak f-t\vphi} tr_{\om_\vphi}\om]^n.
\end{align*}

The Evans-Krylov estimate was proved by Calamai and the second author in Proposition 4.6 and 4.7 in \cite{Calamai-Zheng} with an angle restriction till $\frac{2}{3}$. Generally, we encourage readers for further reading like \cite{MR3264767}, \cite{arXiv:1405.1021}, \cite{JMR}, \cite{arXiv:1307.6375} and the references therein.

\subsection{Existence of K\"ahler-Einstein cone metrics}

A byproduct of Section \ref{Apriori estimates} is a proof of the existence of the K\"ahler-Einstein cone metrics. Ding's functional \cite{MR967024} could be generalised to the conic setting, i.e. for all $\vphi\in \mathcal H^0_\b$,
\begin{align*}
D_\om(\vphi)&=\frac{1}{V}\int_M\vphi\om^n-J_\om(\vphi),\\
F(\vphi)
&=-D_\om(\vphi)-\frac{1}{V}\int_Mf_\om\om^n
+\log(\frac{1}{V}\int_Me^{-\vphi+f_\om}\om^n).
\end{align*}
\begin{thm}
When the conic Ding functional is proper, i.e. there are two positive constants $A$ and $B$ such that for all $\vphi\in \mathcal H_\b^0$,
\begin{align*}
F_\omega(\vphi) \geq A I_\omega(\vphi)- B.
\end{align*}
Then there exits a K\"ahler-Einstein cone metric.
\end{thm}	
\begin{proof} 
When we assume that the conic Ding functional is proper, the $I_\om-J_\om$ functional is bounded along the continuity path. And then the uniform estimates in the sections above i.e. the zero order estimate in Section \ref{Zero order estimate title} and the higher order estimate in Section \ref{Higher order estimate} could be applied to the path. Thus the existence follows from the continuity method.
\end{proof}
The notion of the properness for the smooth K\"ahler metrics was introduced in Tian \cite{MR1471884}.
\subsection{Choosing the automorphism}\label{Choosing the automorphism}

We are given a K\"ahler cone metric $\om$ and an orbit $\mathcal O$ in the space of K\"ahler-Einstein cone metrics.
Let $\theta$ be a K\"ahler-Einstein cone metric in the orbit $\mathcal O$. 
Then there exists $\la_\theta \in C^{2,\alpha}_\beta$ such that 
\begin{align*}
\theta=\om_{\la_\theta}=\om+i\p\bar\p \la_\theta.
\end{align*}

We minimise the following functional over the orbit, i.e. for any $\theta\in \mathcal O$,
\begin{align*}
E(\theta;\om)&:=I(\om,\theta)-J(\om,\theta)\\
&=-\int_M \la_\theta \theta^n.
\end{align*}
\begin{prop}
$E(\cdot;\om)$ has a minimiser $\theta$ over the orbit $\theta\in \mathcal O$ such that $\theta=\om_{\lambda_\theta}$ and $\la_\theta\in C^{2,\alpha}_\beta$.
\end{prop}
\begin{proof}
In order to apply the variational direct method, it suffices to prove the level set of the $E\leq r$ is bounded. That follows from the apriori estimates including the zero order estimate (Proposition \ref{Zero order estimate}) and the higher order estimates (Section \ref{Higher order estimate}). The former is true since $I-J$ is bounded and over the whole orbit all metrics have the same Sobolev constant. While, the latter holds since we have Ricci lower bound of each K\"ahler-Einstein cone metric in $\theta\in \mathcal O$.
\end{proof}
From now on, we use $\theta$ to denote the minimiser.
We also denote $$L_\theta=\tri_\theta+id.$$
\begin{lem}\label{propeetyofgauge}
Let $\theta=\om_{\la_\theta}$ be a minimiser of the function $E$. Then we have that
\begin{itemize}
\item $\forall u\in Ker(L_\theta)$, $\int_M\la_\theta \cdot u\cdot\theta^n=0$,
\item \begin{align*}
D^2 E_{\theta}(u,v)&=\int_M(1+\frac{1}{2}\tri_\theta\la_\theta)\cdot u\cdot v\cdot\theta^n\\
&=\int_M uv-[uv-<\p u,\p v>_\theta]\cdot\la_\theta\cdot \theta^n.
\end{align*}
\end{itemize}
\end{lem}
\begin{proof}
We let $\sigma(t)$ be the one-parameter subgroup generated by the real part of the holomorphic vector field defined by $\uparrow^\theta\bar{\p} u$, which follows from Section \ref{sub2}. I.e. 
$$\sigma(t)=exp(t \Re \uparrow^\theta\bar{\p} u).$$
Then there exists $\rho(t)\in C^{2,\alpha}_\beta$ such that 
\begin{align*}
\sigma^\ast(t)\theta=\theta+i\p\bar\p\rho(t)
\end{align*}
or 
\begin{align*}
\sigma^\ast(t)\theta=\om+i\p\bar\p(\la_\theta+\rho(t)).
\end{align*}
The potential $\rho(t)$
satisfies $$\rho(t=0)=0, \quad \frac{d}{dt}|_{t=0}\rho(t)=u$$ and the K\"ahler-Einstein cone equation
\begin{align*}
\om_{\la_\theta+\rho(t)}^n=\theta^n e^{-\rho(t)}.
\end{align*}	
Differentiating this equation on the both sides on $t$, 
\begin{align}\label{1stdie}
(\tri_{\la_\theta+\rho(t)}+1)\frac{d\rho(t)}{dt}=0.
\end{align}
 and taking $t=0$, we have the linearisation equation at $t=0$,
\begin{align*}
(\tri_{\theta}+1)u=0.
\end{align*} 
Integrating over $M$ with respect to $\theta$ we arrive at
\begin{align}\label{uaverage}
\int_{M} (\tri_{\theta}+1)u \theta^n=0.
\end{align} 
Since $u\in C^{2,\a}_\b$, using integration by parts (Lemma 2.11 in \cite{Calamai-Zheng}), we have 
\begin{align}\label{uaverage}
\int_M u \theta^n=0,
\end{align}

Now the $E$-functional of $\sigma^\ast(t)\theta$ becomes
\begin{align*}
E(\sigma^\ast(t)\theta;\om)=-\int_M\la_\theta+\rho(t)\om^n_{\la_\theta+\rho(t)}.
\end{align*}
Differentiating its both sides on $t$
\begin{align*}
&\frac{d}{dt}E(\sigma^\ast(t)\theta;\om)\nonumber\\
&=-\int_M \frac{d}{dt}\rho(t)\om_{\la_\theta+\rho(t)}^n
-\int_M (\la_\theta+\rho(t))\tri_{\la_\theta+\rho(t)}\frac{d}{dt}\rho(t)\om_{\la_\theta+\rho(t)}^n,
\end{align*}
using \eqref{1stdie}
\begin{align}\label{1stt}
=-\int_M \frac{d}{dt}\rho(t)\om_{\la_\theta+\rho(t)}^n
+\int_M (\la_\theta+\rho(t))\frac{d}{dt}\rho(t)\om_{\la_\theta+\rho(t)}^n,
\end{align}
evaluating at $t=0$ 
\begin{align*}
\frac{d}{dt}E(\sigma^\ast(t)\theta;\om)\vert_{t=0}
=-\int_M u\theta^n
+\int_M \la_\theta  u\theta^n,
\end{align*}
applying the identity \eqref{uaverage} above,
we obtain the RHS equals
\begin{align*}
\int_M \la_\theta u\theta^n.
\end{align*}
Thus the the first identity follows for any $u\in Ker(L_\theta)$, i.e. $\tri_\theta u=-u$.

We further denote $\sigma(s)$ the one-parameter subgroup generated by the holomorphic potential $v$, i.e. $$\sigma(s)=exp(\Re \uparrow^\theta\bar{\p} v).$$ 
Let $\sigma(s,t)=\sigma(s)\sigma(t)$, then there exists K\"ahler cone potential $\rho(s,t)$ such that 
\begin{align*}
\sigma^\ast(s,t)\theta=\theta+i\p\bar\p\rho(s,t)
\end{align*}
or 
\begin{align*}
\sigma^\ast(s,t)\theta=\om+i\p\bar\p(\la_\theta+\rho(s,t)).
\end{align*}
Thus, $\rho(s=0,t=0)=0$, 
$$\frac{d}{dt}|_{t=0}\rho(s,t)=u\text{ and }\frac{d}{ds}|_{s=0}\rho(s,t)=v.$$ Differentiating \eqref{1stdie} on $s$ again,
\begin{align}\label{2nddie}
(\tri_{\la_\theta+\rho(s,t)}+1)\frac{\p^2\rho(s,t)}{\p s\p t}
=<\p\bar\p\frac{\p\rho(s,t)}{\p s},\p\bar\p\frac{\p\rho(s,t)}{\p t}>_{\la_\theta+\rho(s,t)}.
\end{align}	
Then setting $t=s=0$ and using the second formula in \lemref{lem-identity}, we get
\begin{align}\label{partialstrho}
\frac{\p^2\rho(s,t)}{\p s\p t}|_{_{s=t=0}}
= \langle \d u, \d v \rangle_{\theta} 
= \langle \d v, \d u \rangle_{\theta} 
\end{align}
modulo $Ker(\tri_{\theta}+1)$.
Similar to \eqref{1stt}, 
\begin{align*}
&\frac{\p}{\p t}E(\sigma^\ast(s,t)\theta;\om)\\
&=-\int_M \frac{\p}{\p t}\rho(s,t)\om_{\la_\theta+\rho(s,t)}^n
+\int_M (\la_\theta+\rho(s,t))\frac{\p}{\p t}\rho(s,t)\om_{\la_\theta+\rho(s,t)}^n.
\end{align*}
We differentiate it again on $s$,
\begin{align*}
&\frac{\p^2 E(\sigma^\ast(s,t)\theta;\om)}{\p s\p t}\\
&=-\int_M \frac{\p^2}{\p s\p t}\rho(s,t)\om_{\la_\theta+\rho(s,t)}^n
-\int_M \frac{\p}{\p t}\rho(s,t) \tri_{\la_\theta+\rho(s,t)}\frac{\p}{\p s}\rho(s,t) \om_{\la_\theta+\rho(s,t)}^n\\
&+\int_M \frac{\p}{\p s}\rho(s,t) \frac{\p}{\p t}\rho(s,t)\om_{\la_\theta+\rho(s,t)}^n
+\int_M (\la_\theta+\rho(s,t))\frac{\p^2}{\p s\p t}\rho(s,t)\om_{\la_\theta+\rho(s,t)}^n\\
&+\int_M (\la_\theta+\rho(s,t))\frac{\p}{\p t}\rho(s,t) \tri_{\la_\theta+\rho(s,t)}\frac{\p}{\p s}\rho(s,t) \om_{\la_\theta+\rho(s,t)}^n
\end{align*}
and then evaluate at $s=t=0$,
\begin{align*}
&\frac{\p^2 E(\sigma^\ast(s,t)\theta;\om)}{\p s\p t}|_{s=t=0}\\
&=-\int_M \frac{\p^2}{\p s\p t}\rho(s,t)|_{s=t=0} \theta^n
-\int_Mu \tri_{\la_\theta} v \theta^n\\
&+\int_M v u\theta^n
+\int_M \la_\theta\frac{\p^2}{\p s\p t}\rho(s,t)|_{s=t=0} \theta^n\\
&+\int_M \la_\theta u \tri_{\la_\theta} v \theta^n\\
&=-\int_M \frac{\p^2}{\p s\p t}\rho(s,t)|_{s=t=0} \theta^n+2\int_Muv\theta^n\\
&+\int_M\la_\theta \frac{\p^2\rho(s,t)}{\p s\p t}|_{s=t=0} \theta^n
+\int_M \la_\theta u \tri_\theta v \theta^n.
\end{align*}
Recall that both $u$ and $v$ are in $Ker(L_\theta)$, i.e. $\tri_\theta u=-u$ and $\tri_\theta v=-v$.
Using \eqref{partialstrho}, we rewrite the third term to be
\begin{align*}
\int_M\la_\theta  \langle \d u, \d v \rangle_{\theta}  \theta^n.
\end{align*}
And similarly,
the first term is reduced to
\begin{align*}
-\int_M \frac{1}{2}[ \langle \d u, \d v \rangle_{\theta} + \langle \d v, \d u \rangle_{\theta} ] \theta^n.
\end{align*}
The integration by parts (Lemma 2.11 in \cite{Calamai-Zheng}) further implies that the first term could also be transformed as
\begin{align*}
-\int_M uv \theta^n.
\end{align*}
The fourth term becomes, 
\begin{align*}
-\int_M \la_\theta  uv \theta^n.
\end{align*}
Thus, adding these terms together, we arrive at
\begin{align*}
RHS=\int_Muv\theta^n+\int_M\la_\theta  \langle \d u, \d v \rangle_{\theta} \theta^n
-\int_M \la_\theta  uv \theta^n.
\end{align*}
In conclusion, the second identity in the lemma follows directly by integration by parts.
\end{proof}

\subsection{Bifurcation at $t=1$}\label{Bifurcation}
The following existence, uniqueness and regularity of the linear equation with respect to the K\"ahler cone metrics in Calamai-Zheng \cite{Calamai-Zheng} is fundamental when applying the implicit function theorem. The general linear elliptic equation was considered,
\begin{equation}\label{linear equ ge}
  \left\{
   \begin{array}{rl}
Lv&=g^{i\bar j}v_{i\bar j}+b^iv_i+cv=f+\p_ih^i \text{ in } X\setminus D,\\
v&=v_0 \text{ on } \p X
   \end{array}
  \right.
\end{equation}
in the pair $(X,D)$.
Here $\p X$ is the boundary of $X$, $g^{i\bar j}$ is the inverse matrix of a $C^{\a}_\b$ K\"ahler cone metric $\om$ and suitable conditions on the coefficients $ b^i,c, f,  h^i, v_0$ are given. Note that both $b^i$ and $\p_i h^i$ are understood as vectors (not functions).

In \cite{Calamai-Zheng}, the general linear elliptic equation was solved on the manifolds with boundary and the proof used Schauder estimate in Donaldson \cite{MR2975584}. 
But here we only need the theory on the manifolds without boundary and the coefficients 
\begin{align}\label{linear equ ge co}
b^i=h^i=0, c=1 \text{ and }  f\in C^{\a}_\b.
\end{align}
I.e.
\begin{equation}\label{linear equ ge here}
Lv=g^{i\bar j}v_{i\bar j}+v=f\text{ in } M.
\end{equation}
The following is from Proposition 5.21 in \cite{Calamai-Zheng}.

\begin{prop}\label{implicit function theorem}
There exists a $C^{2,\a}_\b$ solution of \eqref{linear equ ge here} with data as \eqref{linear equ ge co}.
\end{prop}

In later application, the coefficient $g$ is a K\"ahler-Einstein metric. So the kernel of this linear equation \eqref{linear equ ge here} generates a holomorphic vector field as proved in Section \ref{sub2}. 

Now we continue our proof of the bifurcation.
\begin{prop}\label{bifurcationat1}
Let $\theta=\om_{\la_\theta}$ be the minimiser of $E(\cdot;\om)$ and be the end point of the continuity path i.e. $\vphi(1)=\la_\theta$. If Hessian of $E(\cdot;\om)$ at $\theta$ is strictly positive definite, i.e.$$D^2 E_{\theta}(u,u)\geq \eps \int_M u^2 \theta^n$$ for some $\eps>0$, then there exists $\delta>0$ such that the continuity path is solvable for any $t\in(1-\delta,1]$.
\end{prop}
\begin{proof}
We now write the continuity path as the fully non-linear operator from 
$C^{2,\a}_\b$ to $C^{\a}_\b$,
\begin{align*}
\Phi(t,\vphi)=\log\frac{\om_\vphi^n}{\om^n}+t\vphi-\mathfrak f.
\end{align*}

We denote $H_\theta$ the kernel space of the linearisation operator $$L_\theta=\tri_\theta + id.$$ The whole space $C^{2,\a}_\b$ is decomposed into the direct sum of $H_\theta$ and its orthogonal space $H^\bot_\theta$. From \lemref{propeetyofgauge}, $\la_\theta\in H^\bot_\theta$.

The path $\vphi-\la_\theta$ is then decomposed into
\begin{align*}
\vphi-\la_\theta=\vphi^\parallel+\vphi^\bot.
\end{align*}
While, we let $P$ denote the projection from $C^{2,\a}_\b$ to $H_\theta$ and decompose the linear operator $\Phi$ into two parts. 

We first consider the vertical part,
\begin{align*}
\Phi^\bot(t,\vphi^\parallel,\vphi^\bot)=(1-P)[\log\frac{\om_{\la_\theta+\vphi^\parallel+\vphi^\bot}^n}{\om^n}-\mathfrak f]+t\cdot(\la_\theta+\vphi^\bot).
\end{align*}
It vanishes at $$(t,\vphi^\parallel,\vphi^\bot)=(1,0,0),$$ since $\theta$ is a K\"ahler-Einstein cone metric. Meanwhiles, its derivative on $\vphi^\bot$ at $(1,0,0)$ is for any $u\in H^\bot_\theta$,
\begin{align*}
\delta_{\vphi^\bot}\Phi^\bot\vert_{(1,0,0)}(u)=\tri_\theta u+ u,
\end{align*}
which is invertible from $H^\bot_\theta$ to itself, according to the existence theorem (Proposition \ref{implicit function theorem}). Therefore, we are able to use the implicit function theorem on $C^{2,\a}_\b$ space to conclude that there is small neighbourhood $U$ near $$(t,\vphi^{\parallel})=(1,0)$$ such that 
\begin{align*}
\vphi^\bot: U \subset(1-\tau,1]\times H_\theta &\rightarrow H_\theta^\bot,\\
(t, \vphi^\parallel)&=\vphi^\bot(t, \vphi^\parallel)
\end{align*} solves \begin{align}\label{Phibot}\Phi^\bot(t,\vphi^\parallel,\vphi^\bot(t,\vphi^{\parallel}))=0
\end{align} when $(t,\vphi^{\parallel})\in U$. Moreover, $$\vphi^\bot(1,0)=0.$$

We further compute the full derivative of $\Phi^\bot$ at $(t,\vphi^{\parallel})=(1,0)$. Let $\vphi^\parallel(s)\in U$ with parameter $0\leq s\leq 0$ and 
\begin{align}\label{vphi(s)}
\vphi^\parallel(s=0)=0,\quad \frac{\p\vphi^\parallel(s)}{\p s}\vert_{s=0}=u\in H_\theta.
\end{align}
We have
\begin{align*}
0=\frac{\p \Phi^{\bot}}{\p s}=(1-P)[\tri_\vphi (\frac{\p\vphi^\parallel}{\p s}+\delta_{\vphi^\parallel}\vphi^\bot (\frac{\p\vphi^\parallel}{\p s}))+ t \cdot \delta_{\vphi^\parallel}\vphi^\bot (\frac{\p\vphi^\parallel}{\p s})]
\end{align*}
and at $s=0$, $t=1$, using $L_\theta u=0,$ 
\begin{align*}
0=\frac{\p \Phi^{\bot}}{\p s}
&=(1-P)[-u+\tri_\theta (\delta_{\vphi^\parallel}\vphi^\bot\vert_{(1,0)} (u))+   \delta_{\vphi^\parallel}\vphi^\bot\vert_{(1,0)} (u)]\\
&=(1-P)[L_\theta (\delta_{\vphi^{\parallel}} \vphi^{\bot}\vert_{(1,0)}(u) )].
\end{align*}
Since both the imagine of $1-P$ and $L_\theta$ are in $H_\theta^\bot$, and 
\begin{align*}
\delta_{\vphi^{\parallel}} \vphi^{\bot}: U\subset(1-\tau,1]\times H_\theta &\rightarrow H_\theta^\bot,
\end{align*}
we conclude that at $t=1$ and $\vphi^\parallel=0$,
\begin{align}\label{parallelbot}
\delta_{\vphi^{\parallel}} \vphi^{\bot}\vert_{(1,0)}(u) &=0, \forall u\in H_\theta.
\end{align}
Meanwhile, we differentiate \eqref{Phibot} on $t$,
\begin{align*}
\frac{\p \Phi^{\bot}}{\p t} &= (1-P)\tri_\vphi\frac{\p \vphi^{\bot}}{\p t}+\la_\theta+\vphi^{\bot}+t\frac{\p \vphi^{\bot}}{\p t}=0.
\end{align*} and evaluate at $t=1$ and $\vphi^\parallel=0$,
\begin{align}\label{vphibottheta}
\frac{\p \Phi^{\bot}}{\p t}\vert_{(1,0)} &= (\tri_\theta+1)\frac{\p \vphi^{\bot}}{\p t}\vert_{(1,0)}+\la_\theta=0.
\end{align}

We next consider the horizontal operator on the finite dimensional space $H_\theta$,
\begin{align*}
\Phi^\parallel(t,\vphi^\parallel)=P[\log\frac{\om_{\la_\theta+\vphi^\parallel+\vphi^\bot(t,\vphi^{\parallel})}^n}{\om^n}-\mathfrak f]+t\cdot\vphi^\parallel.
\end{align*} Then, at $t=1$ and $\vphi^\parallel=0$,
\begin{align*}\frac{\p \Phi^{\parallel}}{\p t}\vert_{(1,0)}=(P[\tri_\vphi\frac{\p \vphi^{\bot}}{\p t}]+\vphi^\parallel)\vert_{(1,0)} =0.
\end{align*}
Also, $\Phi^\parallel$ vanishes at $t=1$ for any $\vphi^\parallel\in H_\theta$, i.e.
\begin{align*}
\Phi^\parallel(1,\vphi^\parallel)=0,
\end{align*}
since at $t=1$, all K\"ahler-Einstein cone metrics are the solution of the nonlinear equation. Then we consider the modified functional
\begin{align*}
\tilde \Phi^\parallel(t,\vphi^\parallel)=\frac{\Phi^\parallel(t,\vphi^\parallel)}{t-1}.
\end{align*}
We could see that as $t\rightarrow 1$, 
\begin{align*}
\tilde \Phi^\parallel(1,\vphi^\parallel)=\frac{\p \Phi^{\parallel}}{\p t}=P[\tri_\vphi\frac{\p\vphi^\bot}{\p t}]+\vphi^{\parallel}.
\end{align*}
Again, we use the family of $\vphi^\parallel(s)$ defined in \eqref{vphi(s)}, its derivative on $\vphi^\parallel$ is
\begin{align}
\frac{\p}{\p s}\tilde \Phi^\parallel(1,\vphi^\parallel(s))
=-P<\p\bar\p[ \frac{\p\vphi^\parallel}{\p s}+ \delta_{\vphi^\parallel}\vphi^{\bot}(\frac{\p\vphi^\parallel}{\p s})],\p\bar\p\frac{\p\vphi^\bot}{\p t}>+\frac{\p\vphi^\parallel}{\p s}.
\end{align}
Then let this derivative evaluates at $s=0$ and use \eqref{parallelbot} to the first term, we have for any $u\in H_\theta$,
\begin{align}\label{bifurcationlinearequ}
(\delta_{\vphi^\parallel}\tilde \Phi^\parallel)\vert_{(1,0)}(u)
&=(\delta_{\vphi^\parallel}\frac{\p \Phi^{\parallel}}{\p t})\vert_{(1,0)}(u)\\
&=-P<\p\bar\p u,\p\bar\p \frac{\p\vphi^{\bot}}{\p t}\vert_{(1,0)}>_\theta+u\nonumber.
\end{align}
So for any $v\in H_\theta$,
\begin{align*}
(\delta_{\vphi^\parallel}\tilde \Phi^\parallel \vert_{(1,0)}(u), v)_{L^2(\theta)}
=\int_M uv-v<\p\bar\p u,\p\bar\p \frac{\p \vphi^{\bot}}{\p t}\vert_{(1,0)}>_\theta \theta^n.
\end{align*}
We apply \eqref{ibp222}, \eqref{vphibottheta} and the proof of \lemref{propeetyofgauge}, 
\begin{align*}
RHS
&=\int_M uv+[uv-<\p u,\p v>_\theta](\tri_\theta+1)\frac{\p}{\p t}\vphi^\bot\vert_{(1,0)} \theta^n\\
&=\int_M uv-[uv-<\p u,\p v>_\theta]\cdot\la_\theta\cdot \theta^n\\
&=D^2 E_{\theta}(u,v).
\end{align*}

Finally, from \lemref{propeetyofgauge} 
\begin{align*}
D^2 E_{\theta}(u,v)=\int_M(1+\frac{1}{2}\tri_\theta\la_\theta)\cdot u\cdot v\cdot\theta^n.
\end{align*}
From the assumption that Hessian of $E(\cdot;\om)$ at $\theta$ is strictly positive definite, i.e.
\begin{align*}
D^2 E_{\theta}(u,u)\geq \eps \int_M u^2 \theta^n.
\end{align*}
Since $u\in H_\theta$, i.e. $\tri_\theta u=-u$, noting that $u\in C^{2,\a}_\b$, we get by integration by parts,
\begin{align}
\int_M |\p u|^2_\theta\theta^n=\int_M u^2\theta^n.
\end{align}
Thus the bilinear form $D^2 E_{\theta}$ is coercive on the Hilbert space $H_\theta$ under the norm $W^{1,2}(\theta)$. 
Actually, $H_\theta$ is finite dimensional, since its element is ono-to-one corresponding to the holomorphic vector field and the set of holomorphic vector fields is finite dimensional.
Meanwhile, there exists a constant $C$ depending on $\theta$ such that, $$D^2 E_{\theta}(u,v) \leq C ||u||_{W^{1,2}(\theta)}\cdot||v||_{W^{1,2}(\theta)}$$ for all $u,v\in H_\theta$. Then the Lax-Milgram theorem implies there is a unique weak solution $u_w\in H_\theta$ such that $$D^2 E_{\theta}(u_w,v)=<w,v>_{W^{1,2}(\theta)}$$ for all $w,v\in H_\theta$. The regularity of $u_w$ is achieved by first applying the Hanarck inequality (Proposition 5.12 in \cite{Calamai-Zheng}), then using Donaldson's Schauder estimate.
Therefore, from the relation
$$(\delta_{\vphi^\parallel}\tilde \Phi^\parallel \vert_{(1,0)}(u), v)_{L^2(\theta)}=D^2 E_{\theta}(u,v),$$
the linearisation operator $\delta_{\vphi^\parallel} \tilde \Phi^\parallel \vert_{(1,0)} $ is invertible from $H_\theta$ to itself.

Then we are able to apply the implicit function theorem to $\tilde \Phi^\parallel(t,\vphi^\parallel)$ over $C^{2,\a}_\b$ to find a solution $\vphi^\parallel(t)\in C^{2,\a}_\b$ with $t\in (1-\tau,1]$ such that 
\begin{itemize}
\item $\tilde\Phi^\parallel(t,\vphi^\parallel(t))=0,$
\item $\vphi(1)=0$.
\end{itemize} 
Thus the original nonlinear equation is solved as
\begin{align*}
\Phi^\bot(t,\vphi^\parallel(t),\vphi^\bot(t,\vphi^{\parallel}(t)))=0.
\end{align*}
And moreover, $$\vphi(t)=\la_\theta+\vphi^\parallel(t)+\vphi^\bot(t,\vphi^{\parallel}(t))$$ is the solution to the continuity path on $t\in (1-\tau,1]$ with $$\vphi(1)=\la_\theta.$$
\end{proof}
\subsection{Proof of the main theorem}
\begin{proof}
The proof is paralleling to Bando-Mabuchi \cite{MR946233}. We are given a K\"ahler cone metric $\om$. We assume that there are two orbits $\mathcal O_1$ and $\mathcal O_2$. We minimise $E(\cdot,\om)$ at $\theta_1$ in $\mathcal O_1$, consider the linear segment $$\om_1^\eps=(1-\eps)\om+\eps\theta_1.$$ We let $E^\eps=E(\cdot,\om_1^\eps)$. Since 
\begin{align*}
\theta_1=\om+i\p\bar\p\la_{\theta_1}
=\om_1^\eps+i\p\bar\p\la^\eps_{\theta_1}
=(1-\eps)\om+\eps\theta_1+i\p\bar\p\la^\eps_{\theta_1},
\end{align*} we have $\la_{\theta_1}^\eps=(1-\eps) \la_{\theta_1}$.
So
\begin{align*}
D^2 E_{\theta_1}(u,v)
&=\int_M(1+\frac{1}{2}\tri_{\theta_1} \la_{\theta_1}^\eps)\cdot u\cdot v\cdot\theta_1^n\\
&=(1-\eps)\int_M(1+\frac{1}{2}\tri_{\theta_1} \la_{\theta_1})\cdot u\cdot v\cdot\theta_1^n
+\eps\int_M u\cdot v\cdot\theta_1^n.
\end{align*}
Then
\begin{align*}
D^2 E^\eps_{\theta_1}(u,u)=(1-\eps)D^2 E_{\theta_1}(u,u) +\eps\int_M u^2\theta_1^n>0.
\end{align*}
Here $D^2 E_{\theta_1}(u,u) $ is non-negative, since $\theta_1$ is the global minimiser in $\mathcal O_1$.

Then we minimise $E(\cdot;\om^\eps_1)$ at $\theta_2$ in $\mathcal O_2$, consider the linear segment $$\om_2^\eps=(1-\eps)\om_1^\eps+\eps\theta_2$$ and let $E_2^\eps=E(\cdot,\om_2^\eps)$.  Again, $D^2 (E_2^\eps)_{\theta_2}(u,u)$ is also strictly positive definite. 

So, choose $\om_2^\eps$ to be sufficient close to $\om^\eps_1$ and choose $\theta^\eps_1$ to be also close to $\theta_1$.  Then $D^2( E_2^\eps)_{\theta_1^\eps}(u,u)$ is again strictly positive definite. 

We now are able to construct two continuity paths connecting $\om_2^\eps$ to both $\theta_1^\eps$ and $\theta^2$, according to the bifurcation (Proposition \ref{bifurcationat1}), the openness (Proposition \ref{weak unique}) and the apriori estimates (Section \ref{Apriori estimates}). But again, the openness (Proposition \ref{weak unique}) implies the solution has to be unique so $\mathcal O_1=\mathcal O_2$.
\end{proof}


\end{document}